\documentclass[12pt]{article}
\usepackage{amsmath,amssymb,amsthm}

 \newcommand{\beqn}{\begin{eqnarray}}
 \newcommand{\eeqn}{\end{eqnarray}}
 \newcommand{\be}{\begin{equation}}
 \newcommand{\ee}{\end{equation}}
 \newcommand{\ba}{\begin{array}}
 \newcommand{\ea}{\end{array}}
 
 \newcommand{\pa}{\partial}
 \newcommand{\re}{\ref}
 \newcommand{\ci}{\cite}
 \newcommand{\ds}{\displaystyle}
 \newcommand{\la}{\label}
 \newcommand{\fr}{\frac}

\newcommand{\ov}{\overline}
\newcommand{\ti}{\tilde}

\newcommand{\bP}{{\bf P}}

\newcommand{\cH}{{\cal H}}

\newcommand{\cO}{{\cal O}}

\newcommand{\cT}{{\cal T}}
\newcommand{\cS}{{\cal S}}

\newcommand{\ve}{\varepsilon}

\newcommand{\al}{\alpha}
\newcommand{\ga}{\gamma}

\newcommand{\si}{\sigma}

\newcommand{\Om}{\Omega}

\newcommand{\Si}{\Sigma}
\newcommand{\lam}{\lambda}

\newcommand{\5}{{\hspace{0.5mm}}}

\newcommand\R{{\mathbb R}}

\renewcommand{\theequation}{\thesection.\arabic{equation}}
\renewcommand{\thesection}{\arabic{section}}
\renewcommand{\thesubsection}{\arabic{section}.\arabic{subsection}}
\newtheorem{theorem}{Theorem}[section]

\renewcommand{\thetheorem}{\arabic{section}.\arabic{theorem}}
\newtheorem{definition}[theorem]{Definition}

\newtheorem{lemma}[theorem]{Lemma}
\newtheorem{remark}[theorem]{Remark}

\newtheorem{cor}[theorem]{Corollary}
\newtheorem{pro}[theorem]{Proposition}

\newcommand{\bd}{\begin{definition}}
 \newcommand{\ed}{\end{definition}}
\newcommand{\bt}{\begin{theorem}}
 \newcommand{\et}{\end{theorem}}

\newcommand{\bp}{\begin{pro}}
\newcommand{\ep}{\end{pro}}

\newcommand{\bl}{\begin{lemma}}
 \newcommand{\el}{\end{lemma}}
\newcommand{\bco}{\begin{cor}}
 \newcommand{\eco}{\end{cor}}
\newcommand{\br}{\begin{remark} }
 \newcommand{\er}{\end{remark}}

\newcommand{\const}{\mathop{\rm const}\nolimits}

\begin{document}
\begin{titlepage}
\bigskip\bigskip\bigskip

\begin{center}
{\Large\bf
On asymptotic stability of moving kink for  \bigskip\\
relativistic Ginzburg-Landau equation}
\vspace{1cm}
\\
{\large E.~A.~Kopylova}
\footnote{Supported partly by the grants of
DFG, FWF and RFBR.}\\
{\it Institute for Information Transmission Problems RAS\\
B.Karetnyi 19, Moscow 101447,GSP-4, Russia}\\
e-mail:~elena.kopylova@univie.ac.at
\medskip\\
{\large A.~I.~Komech}$^{\mbox{\scriptsize 1,}}\!\!$
\footnote{Supported partly by the Alexander von Humboldt
Research Award.}\\
{\it Fakult\"at f\"ur Mathematik, Universit\"at Wien\\
and Institute for Information Transmission Problems RAS}\\
 e-mail:~alexander.komech@univie.ac.at

\end{center}

\date{}
\vspace{0.5cm}

\begin{abstract}
\noindent
We prove the asymptotic stability of the moving kinks for the nonlinear
relativistic wave equations
in one space dimension
 with a Ginzburg-Landau potential: starting in a small 
neighborhood of the kink,  the solution,
asymptotically in time, is the sum of a uniformly 
moving kink and dispersive part
described by the free Klein-Gordon equation.
The remainder decays in a global energy norm.
Crucial role in the proofs play our recent results on
the weighted energy decay for the Klein-Gordon equations.

\smallskip

\end{abstract}

\end{titlepage}

\setcounter{equation}{0}

\section{Introduction}

There has been wide spread interest in the dynamics of topological excitations 
of classical relativistic field theories \ci{Bais82, Bjorn98}.
These excitations are finite energy solutions which do not decay to
one of 
the true ground 
states because of topological constraints, said differently, these 
excitations are separated by an infinitely high potential barrier from
the ground state.
In our contribution we will study in mathematical detail one of the
simplest 
examples. 
The field, $\psi$, is real valued and defined on the line,
$\psi:\R\to\R$. 
The Hamiltonian function reads
\be\la{Ham}
 {\cal H}(\psi,\pi)=\int_\R\Big[\fr12{|\pi(x)|^2}+\fr12{|\psi'(x)|^2}
 + U(\psi(x))\Big]~dx
\ee
with $\pi$ 
the momentum canonically conjugate to $\psi$ and a smooth potential $U$.
 This leads to the equation
of motion  
\be\la{e}
\ddot\psi(x,t)=\psi''(x,t)+ F(\psi(x,t)),\quad x\in\R,
\ee
where $F(\psi)=-U'(\psi)$.
For an introduction, let us consider the Ginzburg-Landau 
quartic double well potential of the form 
$U(\psi)=(\psi^2-a^2)^2/(4a^2)$. Then 
the topological excitations are defined through $\cH<\infty$ and 
the boundary conditions
\be\la{asc} 
\lim_{x\to\pm\infty}\psi(x,t)\to\pm a
\ee  
 with a fixed $a>0$.
Amongst them there are soliton-like solutions which travel with
constant velocity,
$$
\psi(x,t)
=a\tanh\ga\fr{x-vt-q}{\sqrt{2}}.
$$
where 
 $\ga=1/\sqrt{1-v^2}$ is the Lorentz contraction.
The solitons (\re{sls}) are related by a Lorentz boost, since equation
(\re{e}) is relativistically invariant.
We will consider more general 
 double well
potentials for which
\be\la{U1f}
U(\pm a)=U'(\pm a)=0,~~~~~U''(\pm a)>0,
\ee
and 
\be\la{U1}
U(\psi)>0~~ {\rm for}~~ \psi\in(-a,a),
\ee
similarly to the quartic potential.
In this case the soliton-like solutions also exist,
\be\la{sls}
\psi(x,t)
=s(\ga(x-vt-q)),~~~~~~~~~~~~v,q\in\R,~~~~~|v|<1
\ee
where 
$s(\cdot)$ is a "kink" solution to the 
corresponding stationary equation 
\be\la{steq}
 s''(x) - U'(s(x))=0, ~~~~~~~~s(\pm\infty)=\pm a.
\ee

In general our goal is to clarify the special role of the 
soliton-like solutions (\re{sls})  
as long time asymptotics for any finite energy topological 
excitations satisfying 
(\re{asc}). Namely, if one chooses some arbitrary finite energy
initial state  
 satisfying 
(\re{asc}), one would expect that for $t\to\infty$ the solution separates 
into two pieces: one piece is a finite collection of 
travelling solitons of the form  (\re{sls}) 
and their negatives
with some velocities
$v_j\in(-1,1)$ and the shifts $q_j$ depending in a complicated way on the
initial 
data,
 and the second radiative piece which is a dispersive 
solution to the free Klein-Gordon equation which propagates
to infinity with the velocity $1$. 
Our aim here is to elucidate this general picture by mathematical 
arguments for initial data sufficiently close to a soliton (\re{sls}).

Let us discuss our choice 
of the smooth potentials $U$. 
The condition (\re{U1}) is necessary and sufficient 
for the existence of a finite energy 
static solution $s(x)$ to 
(\re{steq}) when (\re{U1f}) holds.
Indeed, the condition is obviously sufficient.
On the other hand, 
 the "energy conservation" 
\be\la{sec}
(s'(x))^2/2-U(s(x))=E
\ee
and $s(\pm\infty)=\pm a$ imply that  $E=0$.
Therefore, $U(\psi)>0$ for $\psi\in(-a,a)$ since otherwise the boundary conditions 
$s(\pm \infty)=\pm a$ would fail.
As a byproduct, our kink solution is monotone increasing, and
\be\la{kinkmon}
s'(x)>0,~~~~~~~~~x\in\R
\ee
Let us note that only the behavior of $U$ near the
interval $[-a,a]$ is of importance since the solution is expected to be close 
to a soliton.
However, we will assume additionally the potential
to be bounded from below
\be\la{infU}
\inf\limits_{\psi\in\R}U(\psi)>-\infty \quad\quad\quad\quad
\ee
to have a well posed  Cauchy problem for all finite energy initial states.

Summarising, we formulate our first basic condition  
on the potential, for technical reasons
adding a flatness condition.  
\\
{\bf Condition U1}. {\it
 The potential $U$ is a real
 smooth function
which satisfies 
(\re{U1f}), ( \re{U1}), (\re{infU}), and
the
following condition holds with some $m>0$}
\be\la{U11}
~~~~~~~~~U(\psi)=\fr{m^2}2(\psi\mp a)^2+\cO(|\psi\mp a|^{14}),~~~~~~~
\psi\to\pm a
\ee

Let us comment on the condition (\re{U11}) (see also Remark \re{rU1}).
First, the condition means that $U''(-a)=U''(a)$
though we do not need the potential to be reflection symmetric.
We consider the solutions close to the kink, $\psi(x,t)=
s(\ga(x-vt-q))+\phi(x,t)$, with
small perturbations $\phi(x,t)$. For such solution the condition
(\re{U11}) and the asymptotics (\re{asc})
mean that the equation (\re{e}) is almost linear Klein-Gordon equation
for large $|x|$ which is helpful for application
of the dispersive properties. 
Finally, we expect that 
the degree $14$ in (\re{U11})
is technical, and a smaller degree should be sufficent.
Let us note that a similar condition has been introduced in \ci{BP,BS}
in the context of the Schr\"odinger equation.
\medskip

Further we need some assumptions on the spectrum of 
the linearised equation. Let us rewrite the equation (\ref{e})
in the vector form,
\beqn\la{eq}
\left\{ \ba{lll}
 \dot{\psi}(x,t)=\pi (x,t)\\
 ~ &  \\
 \dot{\pi}(x,t)=\psi'' (x,t)+  F(\psi(x,t))
\ea\right|~~~x\in\R
\eeqn
Now the soliton-like solutions (\ref{sls}) become
\be\la{sosol}
Y_{q,v}(t)=(\psi_v(x-vt-q),\pi_v(x-vt-q))
\ee
for $q,v\in\R$ with $|v|<1$, where
\be\la{sol}
\psi_v(x)=s(\gamma x),\quad\pi_v(x)=-v\psi_v'(x).
\ee
The states $S_{q,v}:=Y_{q,v}(0)$ form the solitary manifold
\be\la{soman}
{\cal S}:=\{ S_{q,v}: q,v\in\R, |v|<1 \}.
\ee
The linearized operator near the soliton solution $Y_{q,v}(t)$ is
(see Section \ref{lin-sec}, formula (\re{AA1}))
$$
A_v=\left(\ba{cc}
    v\nabla            &          1   \\
    \Delta-m^2-V_v(y)  &          v\nabla
\ea\right),\quad\nabla=\fr{d}{dx},\quad\Delta=\fr{d^2}{dx^2},
$$
where
\be\la{AH}
V_v(x)=-F'(\psi_v(x))-m^2=U''(\psi_v(x))-m^2.
\ee
By (\ref{s-decay}), we have
\be\la{V-decay}
 V_v(x)\sim C(s(\ga x)\mp a)^{12}\sim Ce^{-12m\ga|x|},\quad x\to\pm\infty.
\ee
since
\be\la{s-decay}
s(x)\mp a \sim Ce^{-m|x|},\quad x\to\pm\infty
\ee
by the condition {\bf U1}.

In Section \ref{lin-sec} we show that the spectral properties of the operator $A_v$
are determined by the corresponding properties of its determinant, which is the Schr\"o\-din\-ger operator
\be\la{Hv}
H_{v}=-(1-v^2)\Delta+m^2+ V_v.
\ee
The spectral properties of $H_{v}$ are identical for all $v\in (-1,1)$ since
the relation $V_v(x)=V_0(\ga x)$ implies
\be\la{THT}
H_v=T_v^{-1}H_0T_v,~~{\rm where}~~ T_v:\psi(x)\mapsto\psi(x/\gamma).
\ee
This equivalence manifests the relativistic invariance of the equation (\re{eq}).
The continuous spectrum of the operator $H_{v}$ coincides with $[m^2,\infty)$.
The point $0$ belongs to the discrete spectrum with corresponding eigenfunction
$\psi'_v$.
By (\ref{sol}) and (\ref{kinkmon}) we have
$\psi'_v(x)=\ga s'(\ga x)>0$ for $x\in\R$. Hence, $\psi'_v$ is the groundstate,
and all remaining discrete spectrum is contained in $(0,m^2]$.

For $\al\in\R$,  $p\ge 1$, and  $l=0,1,2,...$ let us denote by
$W^{l,p}_{\al}$,
the  weighted Sobolev space of the functions with the finite norm
$$
\Vert\psi\Vert_{W^{l,p}_\al}=
\sum\limits_{k=0}^l\Vert(1+|x|)^{\al}\psi^{(k)}\Vert_{L^{p}}<\infty
$$
and  $H^l_\al:=W^{l,2}_\al$, so $L^2_\al:=H^0_\al$ are the Agmon's weighted spaces.
\begin{definition} (cf. \cite{JN, M})
A  nonzero solution $\psi\in L^2_{-1/2-0}(\R)\setminus L^2(\R)$
to $H_v\psi=m^2\psi$ is called a resonance.
\end{definition}
Now we can formulate our second basic condition on the
potential.\\
{\bf Condition U2}.
{\it For any $v\in (-1,1)$\\
i) $0$ is only eigenvalue of $H_{v}$. \\
ii) $m^2$ is not
a resonance of $H_v$.
}
\\
We show that  Condition {\bf U2} implies the boundedness of
the resolvent of  the  operator $A_v$ 
in the corresponding weighted Agmon spaces
at the edge points  $\pm im/\ga$
of its continuous spectrum.

Both conditions {\bf U1}, {\bf U2} can be satisfied though it is
non-obvious. 
Let us note that the quartic Ginzburg-Landau potential 
does not satisfy nor 
(\re{U11}) nor the  conditions {\bf U2} i) nor {\bf U2} ii).
We will prove elsewhere that
the corresponding examples of potentials 
satisfying both {\bf U1} and {\bf U2}
can be constructed as smoothened piece-wise quadratic potentials.
\medskip

We now can formulate 
the main result of our paper.
Namely, we will prove
the following asymptotics
\be\la{Si}
 (\psi(x,t), \pi(x,t))\sim
 (\psi_{v_{\pm}}(x-v_{\pm}t-q_{\pm}), \pi_{v_{\pm}}(x-v_{\pm}t-q_{\pm}))
 +W_0(t){\Phi}_{\pm},\quad t\to\pm\infty
\ee
for solutions to  (\re{eq}) with  initial states close to 
a soliton-like solution (\re{sls}).
Here $W_0(t)$ is the dynamical group of the free Klein-Gordon equation,
${\Phi}_\pm$ are the corresponding asymptotic  states, and the remainder
converges to zero $\sim t^{-1/2}$ in the global energy norm
of the Sobolev space $H^1(\R)\oplus L^2(\R)$.
\medskip

Let us comment on previous results in this field.
\medskip\\
$\bullet$ {\it Orbital stability of the kinks}
For 1D relativistic nonlinear Ginzburg-Landau equations (\re{e}) the orbital
stability of the kinks has been proved in \ci{HPW}.
\medskip\\
$\bullet$ {\it The Schr\"odinger equation}
The asymptotics of type (\re{Si}) were established for the first time
by Soffer and Weinstein \ci{SW1,SW2} (see also \ci{PW}) for nonlinear
$U(1)$-invariant Schr\"odinger equation with a potential for small initial
states and sufficiently small nonlinear coupling constant.

The results have been extended by Buslaev and Perelman \ci{BP} to the
translation invariant 1D nonlinear  $U(1)$-invariant Schr\"odinger equation.
The novel techniques \ci{BP} are based on the "separation of variables"
along the solitary manifold and in the transversal directions.
The symplectic projection allows to exclude the unstable directions corresponding
to the zero discrete spectrum of the linearized dynamics.
Similar techniques were developed by Miller, Pego and Weinstein for the
1D modified KdV and RLW equations, \ci{MW96,PW94}.The extensions to higher
dimensions were obtained in \ci{Cu01,KZ07,RSS05,TY02}.
\medskip\\
$\bullet$
{\it Nonrelativistic Klein-Gordon  equations}
The asymptotics of type (\re{Si}) were extended to the nonlinear 3D Klein-Gordon
equations with a potential \ci{SW99}, and for translation invariant system of the
3D Klein-Gordon equation coupled to a particle \ci{IKV05}.
\medskip\\
$\bullet$ {\it Wave front of 3D Ginzburg-Landau  equation}
The asymptotic stability of wave front
 was proved for 3D relativistic Ginzburg-Landau equation
with initial data which differ from the wave front  on a compact set
\cite{Cu}.
The wave front is the solution 
which depends on one space variable only, so it is not a soliton.
The equation differs from the 1D equation (\ref{e}) by the
additional 2D Laplacian. The additional Laplacian improves the dispersive
decay for the corresponding linearized  Klein-Gordon equation
in the continuous spectral space that provides the needed decay for the
transversal dynamics.
\medskip

The proving of the asymptotic stability of the solitons and kinks for
relativistic equations remained an open problem till now.
The investigation
crucially depends on the spectral properties for linearized equation
which are completely unknown for higher dimensions.
For 1D case main obstacle was
the slow decay $\sim t^{-1/2}$ for the free 1D Klein-Gordon equation
(see the discussion in \ci[Introduction]{Cu}).

Let us comment on our approach. We  follow general strategy of
\ci{BP,BS,Cu01,Cu,IKV05,SW99}: symplectic projection onto the
solitary manifold, modulation equations, linearization of
the transversal equations and further Taylor expansion of the
nonlinearity etc. We develop for relativistic equations general scheme
which is common in almost all papers in this area: dispersive
estimates for high energy and low energy components of the solutions 
to linearized equation, virial and
$L^1-L^\infty$ estimates and method of majorants.
However, the corresponding statements and their proofs
in the context of relativistic equations are completely new.

Let us comment on our novel techniques.
\medskip\\
I. The decay $\sim t^{-3/2}$ from Theorem \re{1d} for the linearized 
transversal dynamics relies on our novel approach \cite{KK10,1dkg} to 1D  
Klein-Gordon equation. 
\medskip\\
II. The novel "virial type" estimate (\re{jv1}) 
is relativistic version of the bound
\ci[(1.2.5)]{BS} used in 
\ci{BS} in the context of
 the nonlinear Schr\"odinger equation
(see Remark \re{rU1}). 
\medskip\\
III. We establish an  appropriate relativistic version (\ref{t-as3}) 
of $L^1\to L^\infty$ estimates.\\
Both estimates (\re{jv1}) and (\ref{t-as3}) play crucial role in
obtaining the bounds for the majorants.
\medskip\\
IV. Finally, we give the complete proof of the soliton asymptotics
(\ref{Si}).
In the context of the Schr\"odinger equation,
the proof of the corresponding asymptotics were sketched in \ci{BS}.
\smallskip\\

Our paper is organized as follows. In Section \ref{main-res} we formulate the  main
theorem. In Section \ref{Symproj} we introduce  the symplectic projection onto the
solitary manifold. The linearized equation is defined in Section \ref{lin-sec}.
In Section \ref{Sdec} we split the dynamics in two components: along the solitary
manifold and in the transversal directions. In Section \ref{modsec} the modulation
equations for the parameters of the soliton are displayed. The time decay of the
transversal component is established in Sections \ref{Trans-dec}-\ref{tr-dec}.
Finally, in Section \ref{sol-as} we obtain the soliton asymptotics (\ref{Si}).

\setcounter{equation}{0}
\section{Main results}
\la{main-res}

\subsection{Existence of dynamics}
We consider the Cauchy problem for the Hamilton system (\re{eq})
which we write as
\be\la{Eq}
\dot Y(t)={\cal F}(Y(t)),\quad t\in\R:\quad Y(0)=Y_0.
\ee
Here $Y(t)=(\psi(t), \pi(t))$, $Y_0=(\psi_0, \pi_0)$, and all
derivatives are understood in the sense of distributions.
To formulate our results precisely, let us first  introduce a suitable
phase space for the Cauchy problem (\ref{Eq}).

\begin{definition}\la{def-E}
i) $E_{\al}:=H^1_{\al}\oplus L^2_{\al}$ is the space  of the
states $Y=(\psi,\pi)$ with finite norm
\be\la{nEal}
\Vert \,Y\Vert_{E_{\al}}=
\Vert \psi \Vert_{H^1_\al} +\Vert\pi \Vert_{L^2_\al}<\infty
\ee
ii) The phase space  ${\cal E}:={\cS}+E$, where  $E=E_0$ and
$\cS$ is defined in (\ref{soman}). The metric in ${\cal E}$ is defined as
\be\la{enm}
\rho_{\cal E} (Y_1, Y_2)=\Vert  Y_1-Y_2 \Vert_{E},\quad Y_1,Y_2\in{\cal E}
\ee
iii) $W:=W^{2,1}_0\oplus W^{1,1}_0$ is the space of the states $Y=(\psi,\pi)$
with the finite norm
\be\la{nW}
\Vert \,Y\Vert_{W}=
\Vert \psi \Vert_{W^{2,1}_0} +\Vert\pi \Vert_{W^{1,1}_0}<\infty
\ee
\end{definition}
Obviously, the Hamilton functional (\re{Ham})
is continuous on the phase space ${\cal E}$.
The existence and uniqueness of the solutions to the Cauchy problem
(\ref{Eq}) follows by methods  \cite {Li,Re,St78}:
\begin{pro}
i) For any initial data $Y_0\in{\cal E}$ there exists the unique solution
$Y(t)\in C(\R,\cal E)$ to the problem (\ref{Eq}).\\
(ii) For every $t\in\R$, the map $U(t): Y_0\mapsto Y(t)$ is continuous in
${\cal E}$.\\
(iii)
The  energy is conserved, i.e.
\be\la{E}
{\cal H}(Y(t))= {\cal H}(Y_0),\,\,\,\,\,t\in\R
\ee
\end{pro}
\subsection{Solitary manifold and main result}
Let us consider the solitons (\ref{sol}).
The substitution to (\re{eq}) gives the following stationary equations,
\be\la{stfch}
\left.\ba{rclrcl}
-v\psi'_v(y)&=&\pi_v(y)~,\\\\
-v\pi'_v(y)&=&\psi''_v(y)+F(\psi_v(y))
\ea\right|
\ee
\begin{definition}\la{defsol}
A soliton state is $S(\si):=(\psi_v(x-b),\pi_v(x-b))$, where
$\si:=(b,v)$ with $b\in\R$ and $v\in (-1,1)$.
\end{definition}
Obviously, the soliton solution (\ref{sosol}) admits the representation $S(\si(t))$,
 where
\be\la{sigma}
\si(t)=(b(t),v(t))=(vt+q,v).
\ee
\begin{definition}
A solitary manifold is the set ${\cal S}:=\{S(\si):\si\in\Sigma:=
\R\times (-1,1)$.
\end{definition}
The main result of our paper is the following theorem
\begin{theorem}\la{main}
Let  the conditions {\bf U1} and 
{\bf U2} hold, and
$Y(t)$
be the solution to the Cauchy problem  (\ref{Eq}) with an
initial state $Y_0\in{\cal E}$ which is close to a kink
$S(\si_0)=S_{q_0,v_0}$:
\be\la{close}
  Y_0=S(\si_0)+X_0,\,\,\,\,d_0:=\Vert X_0\Vert_{E_\beta\cap W}\ll 1
\ee
where $\beta >5/2$.
Then for $d_0$ sufficiently small the solution admits the
asymptotics:
\be\la{S}
Y(x,t)=
(\psi_{v_{\pm}}(x-v_{\pm}t-q_{\pm}), \pi_{v_{\pm}}(x-v_{\pm}t-q_{\pm}))
+W_0(t)\Phi_\pm+r_\pm(x,t),\quad t\to\pm\infty
\ee
where $v_{\pm}$ and $q_{\pm}$ are  constants,
$\Phi_\pm\in E$, and $W_0(t)$ is the dynamical group of the
free Klein-Gordon equation , while
\be\la{rm}
\Vert r_\pm(t)\Vert_{E}=\cO(|t|^{-1/2})
\ee
\end{theorem}
It suffices to  prove the asymptotics  (\re{S}) for $t\to+\infty$
since the system (\re{eq}) is time reversible.

\setcounter{equation}{0}
\section{Symplectic projection}
\la{Symproj}
\subsection{Symplectic structure and hamiltonian form}
The system (\re{Eq}) reads as the Hamilton system
\be\la{ham}
\dot Y=J{\cal D}{\cal H}(Y),\,\,\,J:=\left(
\ba{cc}
0 & 1 \\
-1 & 0
\ea
\right),\,\,Y=(\psi,\pi)\in E,
\ee
where ${\cal D}{\cal H}$ is the Fr\'echet derivative of
the Hamilton functional (\re{Ham}).
Let us identify the tangent space of $E$, at every point,
with the space $E$. Consider the symplectic form $\Om$ on $E$ defined  by
\be\la{OmJ}
\Om(Y_1,Y_2)=\langle Y_1,JY_2\rangle,\,\,\,Y_1,Y_2\in E,
\ee
where
$$
\langle Y_1,Y_2\rangle:=\langle\psi_1,\psi_2\rangle+\langle\pi_1,\pi_2\rangle
$$
and $\langle\psi_1,\psi_2\rangle=\ds\int\psi_1(x)\psi_2(x)dx$ etc.
It is clear that the form $\Om$ is non-degenerate, i.e.
$$
\Om(Y_1,Y_2)=0\,\,~\mbox{\rm for every}~~ \,Y_2\in E \,\,\Longrightarrow\,\, Y_1=0.
$$
\begin{definition}
i) The symbol $Y_1\nmid Y_2$ means that $Y_1\in E$,
$Y_2\in E$,
and $Y_1$ is symplectic orthogonal to $Y_2$, i.e. $\Om(Y_1,Y_2)=0$.

ii) A projection operator $\bP: E\to E$
is said to be symplectic
 orthogonal if $Y_1\nmid Y_2$ for $Y_1\in\mbox{\rm Ker}\,\bP$ and
$Y_2\in\mbox {\rm Range}\,\bP$.
\end{definition}

\subsection{Symplectic projection onto solitary manifold}

Let us consider the tangent space ${\cal T}_{S(\si)}{\cal S}$ of
the manifold ${\cal S}$ at a point $S(\si)$. The vectors
\be\la{inb}
\left.\ba{rclrrrrcrcl}
\tau_1=\tau_1(v)&:=&\pa_{b}S(\si)=
(&\!\!\!-\psi'_v(y)&\!\!\!\!,&\!\!\!\!-\pi'_v(y))
\\
\tau_2=\tau_2(v)&:=&\pa_{v}S(\si)=(&\!\!\!\!\pa_{v}\psi_v(y)
&\!\!\!\!,&\!\!\!\!\pa_{v}\pi_v(y))
\ea\right.
\ee
form a basis in ${\cal T}_{S(\si)}{\cal S}$.
Here  $y:=x-b$ is the ``moving frame coordinate''.
Let us stress that the functions $\tau_j$
are always regarded as functions of $y$ rather than those of $x$.
Formula (\re{sol})  implies that
\be\la{tana}
\tau_j(v)\in E_\al, ~~~~v\in (-1,1), ~~~~j=1,2,~~~~~\forall\al\in\R.
\ee
\begin{lemma}\la{Ome}
The symplectic form $\Omega$ is nondegenerate on
the tangent space ${\cal T}_{S(\si)}{\cal S}$
i.e. ${\cal T}_{S(\si)}{\cal S}$ is a symplectic subspace.
\end{lemma}
\begin{proof}
Let us compute the vectors $\tau_1$ and $\tau_2$. Recall that
$\psi_v(y)=s(\gamma y)$ and
$\pi_v=-v\psi'_v(y)=-v\gamma s'(\gamma y)$ with $\gamma=1/\sqrt{1-v^2}$. Then
$$
\tau_1=(\tau_1^1,\tau_1^2)=\Big(-\gamma s'(\gamma y),
~v\gamma^2s''(\gamma y)\Big)
$$
$$
\tau_2=(\tau_2^1,\tau_2^2)=\Big(vy\gamma^3s'(\gamma y),
~-\gamma^3s'(\gamma y)
-v^2y\gamma^4s''(\gamma y)\Big)
$$
Therefore
\be\la{symp}
   \Om(\tau_1,\tau_2)=\langle\tau_{1}^1,\tau_{2}^2\rangle
   -\langle\tau_{1}^2,\tau_{2}^1\rangle=\gamma^4
   \langle s'(\gamma y),s'(\gamma y)\rangle>0
\ee
\end{proof}
Now we show that in a small neighborhood of the soliton
manifold ${\cal S}$ a ``symplectic orthogonal projection''
onto ${\cal S}$ is well-defined. Let us introduce the translations
$T_q:(\psi(x),\pi(x))\mapsto (\psi(x-q),\pi(x-q))$, $q\in\R$.
Note that the manifold ${\cal S}$ is invariant with respect to the
translations.
\begin{definition}
For any  $\ov v<1$ denote by $\Si(\ov v)=\{\si=(b,v):
b\in\R, |v|\le \ov v \}$.
\end{definition}
Let us note that ${\cal S}\in E_\al$ with $\al<-1/2$.
\begin{lemma}\la{skewpro}
Let  $\al<-1/2$ and $\ov v<1$. Then\\
i) there exists a neighborhood ${\cal O}_\al({\cal S})$ of ${\cal S}$ in $E_\al$
and a mapping ${\bf\Pi}:{\cal O}_\al({\cal S})\to{\cal S}$  such that ${\bf\Pi}$
is uniformly continuous on ${\cal O}_\al({\cal S})$ in the metric of $E_\al$,
\be\la{comm}
{\bf\Pi} Y=Y~~\mbox{for}~~ Y\in{\cal S}, ~~~~~\mbox{and}~~~~~
Y-S \nmid \cT_S{\cal S},~~\mbox{where}~~S={\bf\Pi} Y.
\ee
ii) ${\cal O}_\al({\cal S})$ is invariant with respect to the translations
$T_q$, and
\be\la{commut}
{\bf\Pi} T_qY=T_q{\bf\Pi} Y,~~~~~\mbox{for}~~Y\in{\cal O}_\al({\cal S})
~~\mbox{and}~~q\in\R.
\ee
iii) For any $\ov v<1$ there exists an $r_\al(\ov v)>0$ s.t.
$S(\si)+X\in{\cal O}_\al({\cal S})$ if $\si\in\Si(\ov v)$
and $\Vert X\Vert_{E_\al}<r_\al(\ov v)$.
\end{lemma}
\begin{proof} We have to find $\si=\si(Y)$ such that $S(\si)={\bf\Pi} Y$ and
\be\la{ift}
\Om(Y-S(\si),\pa_{\si_j}S(\si))=0,~~~~~~ j=1,2.
\ee
Let us fix an arbitrary $\si^0\in\Sigma$ and note that
the system  (\re{ift}) involves two smooth scalar functions of $Y$.
Then for $Y$ close to $S(\si^0)$, the existence of $\si$
follows by the standard finite dimensional implicit function theorem if we
show that the $2\times 2$ Jacobian matrix with elements $M_{lj}(Y)=
\pa_{\si_l}\Om(Y-S(\si^0),\pa_{\si_j}S(\si^0))$ is non-degenerate
at $Y=S(\si^0)$. First note that all the derivatives exist by  (\re{tana}).
The non-degeneracy holds by Lemma \re{Ome} and the definition  (\re{inb})
since $M_{lj}(S(\si^0))=-\Om(\pa_{\si_l}S(\si^0),\pa_{\si_j}S(\si^0))$.
Thus, there exists some neighborhood ${\cal O}_\al(S(\si^0))$ of $S(\si^0)$
where $\bf\Pi$ is well defined and satisfies (\re{comm}), and the same is true
in the union ${\cal O}'_\al({\cal S})=\cup_{\si^0\in \Si} {\cal O}_\al(S(\si^0))$.
The identity (\re{commut}) holds for $Y,T_qY\in{\cal O}'_\al({\cal S})$,
since the form $\Om$ and the manifold $\cal S$ are invariant with respect to
the translations.
It remains to modify ${\cal O}'_\al({\cal S})$ by the translations: we set
${\cal O}_\al({\cal S})=\cup_{b\in\R}T_b{\cal O}'_\al({\cal S})$. Then the second
statement obviously holds.

The last two statements and the uniform continuity in the first statement
follow by translation invariance and the compactness arguments.
\end{proof}

\medskip

\noindent
We refer to $\bf\Pi$ as  symplectic orthogonal projection onto ${\cal S}$.
\setcounter{equation}{0}
\section{Linearization on the solitary manifold}
\la{lin-sec}
Let us consider a solution to the system (\re{eq}), and split it as the sum
\be\la{dec}
Y(t)=S(\si(t))+X(t)
\ee
where $\si(t)=(b(t),v(t))\in\Sigma$ is an arbitrary smooth function of $t\in\R$.
In detail, denote $Y=(\psi,\pi)$ and $X=(\Psi,\Pi)$.
Then (\re{dec}) means that
\be \la{add}
\left.
\ba{rclrcl}
\psi(x,t)&=&\psi_{v(t)}(x-b(t))+\Psi(x-b(t),t)\\
\pi(x,t)&=&\5\pi_{v(t)}(x-b(t))+\Pi(x-b(t),t)
\ea
\right|
\ee
Let us substitute (\re{add}) to (\re{eq}), and linearize the equations in $X$.
Setting $y=x-b(t)$ which is the ``moving frame coordinate'', we obtain that
\be\la{addeq}
\left.
\ba{rcl}
\dot\psi\!\!\!&=&\!\!\!\dot v\partial_v\psi_v(y)-\dot b\psi'_v(y)+
\dot\Psi(y,t)-\dot b\Psi'(y,t)=\pi_v(y)+\Pi(y,t)
\\\\
\dot\pi\!\!\!&=&\!\!\!\dot v\partial_v\pi_v(y)-\dot b \pi'_v(y)+
\dot\Pi(y,t)-\dot b\Pi'(y,t)
=\psi''_v(y)+\Psi''(y,t)+F(\psi_v(y)+\Psi(y,t))
\ea
\right|
\ee
Using the equations (\re{stfch}),  we obtain from (\re{addeq})
the following equations for the components of the vector $X(t)$:
\be\la{Phi}
\left.\ba{rcl}
\!\!\!\!\dot \Psi(y,t)\!\!\!\!&=&\!\!\!\!\Pi(y,t)+\dot b\Psi'(y,t)+
(\dot b-v)\psi'_v(y)-\dot v \partial_v\psi_v(y)
\\\\
\!\!\!\!\dot \Pi(y,t)\!\!\!\!&=&\!\!\!\!\Psi''(y,t)+\dot b\Pi'(y,t)
+(\dot b-v)\pi'_v(y)-\dot v\partial_v\pi_v(y)+F(\psi_v(y)+\Psi(y,t))
-F(\psi_v(y))
\ea\right|
\ee
We can write the equations (\re{Phi}) as
\be\la{lin}
\dot X(t)=A(t)X(t)+T(t)+{\cal N}(t),\,\,\,t\in\R
\ee
where $T(t)$ is the sum of terms which do not depend on $X$,
and ${\cal N}(t)$ is at least quadratic in $X$.
The linear operator $A(t)=A_{v,w}$ depends on two parameters, $v=v(t)$, and
$w=\dot b(t)$ and can be written in the form
\be\la{AA}
A_{v,w}\left(\ba{c}
\Psi \\ \Pi
\ea\right):=\left(\ba{cc}
w\nabla            & 1      \\
\Delta+ F'(\psi_v) & w\nabla
\ea
\right)\left(\ba{c}
\Psi \\ \Pi
\ea\right)=\left(\ba{cc}
w\nabla            &   1       \\
\Delta-m^2-V_v(y)  &   w\nabla
\ea\right)\left(\ba{c}
\Psi \\ \Pi
\ea\right)
\ee
where
\be\la{V}
V_v(y)= -F'(\psi_v)-m^2
\ee
Furthermore,   $T(t)$ and ${\cal N}(t)={\cal N}(\si,X)$ are given by
\be\la{TN}
T=\left(\ba{c}
(w-v)\psi'_v-\dot v\partial_v\psi_v\\
(w-v)\pi'_v-\dot v\partial_v\pi_v\\
\ea\right),\,\,\,\,
{\cal N}(\si,X)=\left(\ba{c}
0 \\ N(v,\Psi)
\ea\right)
\ee
where $v=v(t)$, $w=w(t)$, $\si=\si(t)=(b(t),v(t))$,  $X=X(t)$, and
\be\la{N2}
 N(v,\Psi)=F(\psi_v+\Psi)-F(\psi_v)-F'(\psi_v)\Psi
\ee
\begin{remark}\la{rT}
{\rm
i)
The term $A(t)X(t)$ in the right hand side of equation  (\re{lin})
is linear  in $X(t)$, and ${\cal N}(t)$ is a {\it high order term} in $X(t)$.
On the other hand, $T(t)$ is a zero order term which does not vanish
at $X(t)=0$ since $S(\si(t))$ generally is not a kink solution if
(\re{sigma}) fails to hold (though $S(\si(t))$ belongs to the solitary manifold).
\\
ii) Formulas (\re{inb}) and (\re{TN}) imply:
\be\la{Ttang}
T(t)=-(w-v)\tau_1-\dot v\tau_2
\ee
and hence $T(t)\in {\cal T}_{S(\si(t))}{\cal S}$, $t\in\R$.
This fact suggests an unstable character of the nonlinear dynamics
{\it along the solitary manifold}.}
\end{remark}
\subsection{Linearized equation}
Here we collect some Hamiltonian and spectral properties of the
operator $A_{v,w}$. First, let us consider the linear equation
\be\la{line}
\dot X(t)=A_{v,w}X(t),~~~~~~~t\in\R
\ee
with arbitrary fixed $v\in (-1,1)$ and $w\in \R$. Let us define
the space $E^+:=H^2(\R)\oplus H^1(\R)$.
\begin{lemma} \la{haml}
i) For any $v\in (-1,1)$ and $w\in \R$ equation (\re{line})
 can be represented as the Hamiltonian system  (cf. (\re{Ham})),
\be\la{lineh}
\dot X(t)=
JD{\cal H}_{v,w}(X(t)),~~~~~~~t\in\R
\ee
where $D{\cal H}_{v,w}$ is the Fr\'echet derivative of the
Hamiltonian functional
\be\la{H0}
{\cal H}_{v,w}(X)=\fr12\int\Big[|\Pi|^2+
|\Psi'|^2+(m^2+V_v)|\Psi|^2]dy+\int\Pi w\Psi' dy
\ee
ii) The energy conservation law holds
for the solutions $X(t)\in C^1(\R,E^+)$,
\be\la{enec}
\cH_{v,w}(X(t))=\const,~~~~~t\in\R
\ee
iii) The skew-symmetry relation holds,
\be\la{com}
\Omega(A_{v,w}X_1,X_2)=-\Omega(X_1,A_{v,w}X_2), ~~~~~~~~X_1,X_2\in E
\ee
\end{lemma}
\begin{proof}
 i) The equation (\re{line}) reads as follows,
\be\la{eql}
\fr{d}{dt}\left(
\ba{c}
\Psi \\ \Pi
\ea
\right)=\left(
\ba{l}
\Pi+w\Psi' \\
\Psi''-(m^2+V_v)\Psi+w\Pi'
\ea
\right)
\ee
The  equations correspond to the Hamilton form since
$$
\Pi+w\Psi'=D_\Pi{\cal H}_{v,w},\,\,\, \Psi''-
(m^2+V_v)\Psi+w\Pi'=-D_\Psi{\cal H}_{v,w}
$$
ii) The energy conservation law
follows by  (\re{lineh}) and the chain rule for the
Fr\'echet derivatives:
\be\la{crF}
\ds\fr d{dt}\cH_{v,w}
(X(t))=\langle D\cH_{v,w}(X(t)),\dot X(t)\rangle=
\langle D\cH_{v,w}(X(t)),J D\cH_{v,w}(X(t))\rangle=0,~~~~~~t\in\R
\ee
since the operator $J$ is skew-symmetric by (\re{ham}), and
$D\cH_{v,w}(X(t))\in E$ for  $X(t)\in E^+$.
\\
iii) The skew-symmetry
holds
since $A_{v,w}X=JD\cH_{v,w}(X)$,
and the linear operator
$X\mapsto D\cH_{v,w}(X)$
is symmetric as the Fr\'echet derivative of a real quadratic form.
\end{proof}
\begin{lemma} \la{ljf}
The operator $A_{v,w}$ acts on the
tangent vectors $\tau=\tau_j(v)$ to the solitary manifold
as follows,
\be\la{Atan}
A_{v,w}[\tau_1]=(v-w)\tau'_1,\,\,\,A_{v,w}[\tau_2]=(w-v)\tau'_2+\tau_1
\ee
\end{lemma}
\begin{proof}
In detail, we have to show that
$$
A_{v,w}\left(
\ba{c}
-\psi'_{v} \\ -\pi'_{v}
\ea
\right)=\left(
\ba{c}
(v-w)\psi''_{v}  \\ (v-w)\pi''_{v}
\ea
\right),\quad
A_{v,w}\left(
\ba{c}
\pa_{v}\psi_{v}  \\ \pa_{v}\pi_{v}
\ea
\right)=\left(
\ba{c}
(w-v)\pa_{v}\psi'_{v}  \\ (w-v)\pa_{v}\pi'_{v}
\ea
\right)+\left(
\ba{c}
-\psi'_{v}  \\ -\pi'_{v}
\ea\right)
$$
Indeed, differentiate the equations (\re{stfch}) in $b_j$
 and $v_j$, and obtain that the derivatives of soliton state in
parameters satisfy the following equations,
\be\la{stinb}
\left.\ba{rclrcl}
-v\psi''_{v}\!\!\!\!&=&\!\!\!\!\pi'_{v}\,,
&
-v\pi''_{v}\!\!\!\!&=&\!\!\!\!\psi'''_{v}+F'(\psi_v)\psi_v'
\\\\
-\psi'_{v}-v\pa_{v}\psi'_{v}\!\!\!\!&=&\!\!\!\!\pa_{v}\pi_{v}\,,
&
-\pi'_{v}- v\pa_{v}\pi'_{v}\!\!\!\!&=&\!\!\!\!
\pa_{v}\psi''_{v}+F'(\psi_v)\pa_{v}\psi_v
\ea\right.
\ee
Then (\re{Atan}) follows from (\re{stinb}) by definition of $A_{v,w}$ in (\re{AA})
\end{proof}
Now we consider the  operator $A_v=A_{v,v}$ corresponding to $w=v$:
\be\la{AA1}
A_{v}:=\left(
\ba{cc}
v\nabla           & 1 \\
\Delta-m^2-V_v(y) & v\nabla
\ea
\right)
\ee
In that case the linearized equation has the following additional specific features.
The continuous spectrum of the operator $A_{v}$ coincides with
\be\la{cC}
\Gamma:=(-i\infty,-im/\ga\,]\cup [im/\ga,~i\infty)
\ee
From (\ref{Atan}) it follows that the tangent vector $\tau_1(v)$  is the zero
eigenvector, and $\tau_{2}(v)$ is the corresponding root vector of the
operator $A_{v}$, i.e.
\be\la{Atanform}
A_{v}[\tau_1(v)]=0,\,\,\,A_{v}[\tau_{2}(v)]=\tau_1(v)
\ee
\begin{lemma}\la{SPP1}
Zero root space of operator $A_v$ is two-dimensional for any $v\in (-1,1)$.
\end{lemma}
\begin{proof}
It suffices to check that the  equation $A_vu(v)=\tau_2(v)$ has no solution
in $L^2\oplus L^2$. Indeed, the equation reads
\be\la{rs}
\left(
\ba{cc}
v\nabla            & 1 \\
\Delta-m^2- V_v(y) & v\nabla \\
\ea
\right)\left(
\ba{c}
u_1 \\ u_2
\ea
\right)=\left(\ba{c}
v\ga^3ys'(\ga y) \\ -\ga^3s'(\ga y)-v^2\ga^4ys''(\ga y)
\ea\right)
\ee
From the first equation we get  $u_2=v\ga^3ys'(\ga y)-v\nabla u_1$.
Then the second equation implies that
\be\la{rs1}
H_vu_1=-\ga^3(1+v^2)s'(\ga y)-2v^2\ga^4ys''(\ga y)
\ee
where $H_v$ is the  Schr\"odinger  operator defined in (\ref{Hv}).
Setting $u_1=-\fr 12v^2\ga^5y^2s'(\ga y)+\tilde u_1$, we reduce the equation to
\be\la{rs2}
H_v\tilde u_1=-\ga^2\psi'_v
\ee
i.e. $\tilde u_1$ is the root function of the operator $H_v$
since $\psi'_v$ is eigenfunction. However, this is impossible since
$H_v$ is selfadjoint operator.
\end{proof}
\begin{lemma}\la{SPP2}
The operator $A_{v}$ has only eigenvalue  $\lam=0$.
\end{lemma}
\begin{proof}
Let us consider the eigenvalues problem for operator $A_{v}$:
$$
\left(
\ba{cc}
v\nabla            & 1 \\
\Delta-m^2- V_v(y) & v\nabla
\ea
\right)\left(
\ba{c}
u_1 \\ u_2
\ea
\right)=\lam\left(\ba{c}
u_1 \\ u_2
\ea\right)
$$
From the first equation we have $u_2=-(v\nabla-\lam)u_1$.
Then the second equation implies that
\be\la{lam}
(H_v+\lam^2-2v\lam\nabla)u_1=0
\ee
Hence, for $v=0$
the operator $A_0$ has only eigenvalue  $\lam=0$ by  Condition {\bf U2} i).

Further, let us consider  the case $v\not=0$.
Taking the scalar product with $u_1$, we obtain
$$
\langle H_vu_1, u_1\rangle+\lam^2\langle u_1, u_1\rangle=0
$$
 Hence, $\lam^2$ is real since the operator $H_v$ is selfadjoint.
 The nonzero eigenvalues
can  bifurcate either from the point $\lam=0$ or from the edge points $\pm im/\ga$
of the continuous spectrum of the operator $A_v$. Let us consider each case separately.

i) The point $\lam=0$ cannot bifurcate since it is isolated, and the zero root space
is two dimensional by Lemma \ref{SPP1}.

ii) The bifurcation from the edge points also is impossible.
Indeed, the bifurcated
eigenvalue $\lam\in (-im/\ga,im/\ga)$ is pure imaginary because  $\lam^2$ is real.
Hence, (\re{lam}) is equivalent to
\be\la{lam1}
\Big(H_v +\ga^2\lam^2\Big)p=0
\ee
where $p(x) =e^{\ga^2v\lam x}u_1(x)\in L^2$ that is forbidden by
Condition  {\bf U2} i) since  $-\ga^2\lam^2\in (0,m^2)$.
\end{proof}
By the same arguments we obtain
\begin{lemma}\la{SPP3}
The equation
\be\la{Hme}
\Big(H_v+\fr {m^2}{\ga^2}\pm\fr{i2vm}{\ga}\nabla\Big)\psi=0
\ee
has no  nonzero solution  $\psi\in L^2_{-1/2-0}$.
\end{lemma}
\begin{proof}
The equation (\re{Hme}) is equivalent to
$$
(H_v-m^2)p=0,~~~{\rm where}~~~ p(x)=e^{\pm iv\ga x}\psi(x)
$$
The last equation has no nonzero solution  $p\in L^2_{-1/2-0}$
by Condition  {\bf U2} ii).
\end{proof}
\subsection{Decay for the linearized dynamics}
Let us consider  the linearized equation
\be\la{lin1}
\dot X(t)=A_vX(t),\,\,\,t\in\R
\ee
where $A_v=A_{v,v}$ is given in (\ref{AA1}) with $V_v$ is defined in (\ref{V}).
\begin{definition}
For $|v|<1$, denote by $\bP_v^d$ the symplectic orthogonal projection
of $E$ onto the tangent space $\cT_{S(\si)}{\cal S}$, and
$\bP_v^c={\bf I}-\bP_v^d$.
\end{definition}
Note that by the linearity,
\be\la{Piv}
\bP_v^dX=\sum p_{jl}(v)\tau_j(v)\Om(\tau_l(v),X),\quad X\in E
\ee
with some smooth coefficients $p_{jl}(v)$.
Hence, the projector $\bP_v^d$, in the variable $y=x-b$,
does not depend on $b$.

Next decay estimates will play the key role in our proofs.
The first estimate follows from  our assumption {\bf U2} by
Theorem 3.15 of  \cite{1dkg} since the condition
of type \cite[(1,3)]{1dkg} holds in our case (see also \cite{KK10}).
\begin{theorem}\la{1d}
Let the condition {\bf U2} hold, and $\beta>5/2$.
Then for any $X\in E_\beta$, the weighted energy decay holds
\begin{equation}\label{t-as1}
\Vert e^{A_vt} \bP_v^c X\Vert_{E_{-\beta}}
\le C(v)(1+t)^{-3/2}\Vert X\Vert_{E_{\beta}},\; t\in\R
\end{equation}
\end{theorem}
\begin{cor}\la{cr1}
For $\sigma>5/2$ and for $X\in E_\beta\cap W$
\begin{equation}\label{t-as3}
\Vert (e^{A_vt}\bP_v^c X)_1\Vert_{L^{\infty}}
\le C(v)(1+t)^{-1/2}(\Vert X\Vert_{W}+\Vert X\Vert_{E_{\beta}}),\; t\in\R
\end{equation}
Here $(\cdot)_1$ stands for the first component of the vector function.
\end{cor}
\begin{proof}
Let us apply the projector $\bP^c_v$ to both sides of (\ref{lin1}):
\be\la{lin2}
\bP_v^c \dot X=A_v\bP^c_v X=A_v^0 \bP^c_vX+{\bf V}_v\bP^c_v X
\ee
where
$$
A_v^0=\left(\ba{cc} v\nabla& 1 \\ \Delta-m^2 & v\nabla\\\ea\right),\quad
{\bf V}=\left(\ba{cc}0& 0 \\ -V_v & 0\\\ea\right)
$$
Hence, the Duhamel representation gives,
\be\la{Dug}
  e^{A_vt}Y= e^{A_v^0t}Y+
  \int\limits_0^t e^{A_v^0(t-\tau)}{\bf V}e^{A_v\tau}Yd\tau,\quad
Y=\bP^c_v X,\quad t >0.
\ee
Let us note, that $ e^{A_v^0t}Z= e^{A_0^0t}T_vZ$, where
$T_vZ(x,t)=Z(x+vt,t)$. Then (\ref{Dug}) reads
\be\la{Dug1}
  e^{A_vt}Y= e^{A_0^0t}T_vY+
  \int\limits_0^t e^{A_0^0(t-\tau)}T_v[{\bf V}e^{A_v\tau}Y]d\tau,\quad t >0
\ee
Applying estimate (265) from \cite{RS3}, the H\"older inequality
and Proposition \ref{1d} we obtain
\beqn\nonumber
\Vert (e^{A_vt} Y)_1\Vert_{L^{\infty}}&\le& C(1+t)^{-1/2}\Vert T_vY\Vert_{W}
+ C\int\limits_0^t (1+t-\tau)^{-1/2}\Vert T_v[V(e^{A_v\tau}Y)_1]\Vert_{W_0^{1,1}}~d\tau\\
\nonumber
&=&C(1+t)^{-1/2}\Vert Y\Vert_{W}
+ C\int\limits_0^t (1+t-\tau)^{-1/2}\Vert V(e^{A_v\tau}Y)_1\Vert_{W_0^{1,1}}~d\tau\\
\nonumber
&\le& C(1+t)^{-1/2}\Vert X\Vert_{W}
+ C\int\limits_0^t (1+t-\tau)^{-1/2}\Vert e^{A_v\tau}\bP^c_v X\Vert_{E_{-\si}}~d\tau\\
\nonumber
&\le& C(1+t)^{-1/2}\Vert X\Vert_{W}
+ C\int\limits_0^t (1+t-\tau)^{-1/2}(1+\tau)^{-3/2}\Vert X\Vert_{E_{\si}}~d\tau\\
\nonumber
&\le& C(1+t)^{-1/2}(\Vert X\Vert_{W}+\Vert X\Vert_{E_{\si}})
\eeqn
\end{proof}
\subsection{Taylor expansion for nonlinear term}
Now let us expand $N(v,\Psi)$ from (\ref{N2}) in the Taylor series
\be\la{NN}
N(v,\Psi)=N_{2}(v,\Psi)+N_{3}(v,\Psi)+...+N_{12}(v,\Psi)+N_R(v,\Psi)
=N_I(v,\Psi)+N_R(v,\Psi)
\ee
where
\be\la{Nj}
N_{j}(v,\Psi)=\fr{F^{(j)}(\psi_v)}{j!} \Psi^j,\quad j=2,...,12
\ee
and $N_R$ is the remainder. By condition {\bf U1} we have
$$
F(\psi)=-m^2(\psi\mp a)+\cO(|\psi\mp a|^{13}),\quad \psi\to\pm a
$$
 Hence, the functions
$F^{(j)}(\psi_v(y))$, $2\le j\le 12$ decrease exponentially  as $|y|\to\infty$
by (\ref{s-decay}) and (\re{sol}). Therefore,
\be\la{NI-est}
\Vert N_I\Vert_{L^2_{\beta}\cap W_0^{1,1}}={\cal R}(\Vert \Psi\Vert_{L^\infty})
\Vert \Psi\Vert_{L^\infty}\Vert\Psi\Vert_{H^1_{-\beta}}
={\cal R}(\Vert \Psi\Vert_{L^\infty})\Vert \Psi\Vert_{L^\infty}
\Vert X\Vert_{E_{-\beta}}
\ee
For the remainder $N_R$ we have
\be\la{NRR}
|N_R|={\cal R}(\Vert \Psi\Vert_{L^\infty})|\Psi|^{13}
\ee
where ${\cal R}(A)$ is a general notation for a positive function
which remains bounded as $A$ is sufficiently small.
\begin{lemma}
The bounds hold
\be\la{NR-est1}
\Vert N_R\Vert_{W_0^{1,1}}={\cal R}(\Vert \Psi\Vert_{L^\infty})
\Vert \Psi\Vert_{L^\infty}^{10}
\ee
\be\la{NR-est}
\Vert N_R\Vert_{L^2_{5/2+\nu}}={\cal R}(\Vert \Psi\Vert_{L^\infty})
(1+t)^{4+\nu}\Vert \Psi\Vert_{L^\infty}^{12},\quad 0<\nu<1/2
\ee
\end{lemma}
\begin{proof}
{\it Step i)} By the Cauchy formula,
\be\la{Cf}
\hat N_{R}(x,t)=N_{12}(x,t)+N_{R}(x,t)=\frac{\Psi^{12}(x,t)}{(12)!}
\int\limits_0^1(1-\rho)^{11}F^{(12)}(\psi_v+\rho\Psi(x,t))d\rho
\ee
Therefore,
$$
\Vert \hat N_{R}\Vert_{L^1}={\cal R}(\Vert \Psi\Vert_{L^\infty})\int|\Psi|^{12}dx
={\cal R}(\Vert \Psi\Vert_{L^\infty})\Vert\Psi\Vert_{L^\infty}^{10}
\Vert\Psi\Vert^2_2
={\cal R}(\Vert \Psi\Vert_{L^\infty})\Vert)\Vert\Psi\Vert_{L^\infty}^{10}
$$
since $\Vert\Psi\Vert_{L^2}\le C(d_0)$ by the results of \cite{HPW}.\\
Differentiating (\ref{Cf}), we obtain
$$
\hat N'_{R}=\frac{\Psi^{12}}{(12)!}\int\limits_0^1(1\!-\!\rho)^{11}(\psi_v'+\rho\Psi')
F^{(13)}(\psi_v+\rho\Psi)d\rho
+\frac{\Psi^{11}\Psi'}{(11)!}\!\int\limits_0^1(1\!-\!\rho)^{11}F^{(12)}
(\psi_v+\rho\Psi)d\rho
$$
Hence,
$$
\Vert \hat N'_{R}\Vert_{L^1}={\cal R}(\Vert \Psi\Vert_{L^\infty})
\Big[\Vert \Psi\Vert_{L^\infty}^{12}
+\Vert\Psi\Vert_{L^\infty}^{10}\int|\Psi(x)\Psi'(x)|dx\Big]
\le{\cal R}(\Vert \Psi\Vert_{L^\infty})\Vert\Psi\Vert_{L^\infty}^{10}
$$
since
$\ds\int|\Psi(x)\Psi'(x)|dx\le\Vert\Psi\Vert_{L^2}
\Vert\Psi'\Vert_{L^2}\le C(d_0)$.
Finally, note that
$$
\Vert\Psi_{12}\Vert_{W_0^{1,1}}={\cal R}(\Vert \Psi\Vert_{L^\infty})
\Vert \Psi\Vert_{L^\infty}^{10}
$$
Then (\ref{NR-est1}) follows.\\
{\it Step ii)}
The bound (\ref{NRR}) implies
$$
\Vert N_R\Vert_{L^2_{5/2+\nu}}={\cal R}(\Vert \Psi\Vert_{L^\infty})
\Vert\Psi\Vert_{L^\infty}^{12}\Vert \Psi\Vert_{L^2_{5/2+\nu}}
$$
We will prove  in Appendix B that
\be\la{jv1}
\Vert\Psi(t)\Vert_{L^2_{5/2+\nu}}\le C(d_0)(1+t)^{4+\nu}
\ee
Then (\ref{NR-est}) follows.
\end{proof}
\br\la{rU1}
{\rm
Our choice of the degree $14$ 
in the condition (\re{U11})
is due to the competition between
the factors in the estimate 
 (\re{NR-est}) for the remainder. Namely, the factor $(1+t)^{4+\nu}$ with $\nu<1/2$
comes from 
the virial type estimate    (\re{jv1}) describing the expansion 
of the support for the perturbation of the kink. 
On the other hand, $\Vert\Psi\Vert^{12}_{L^\infty}\sim t^{-6}$
by the crucial decay estimate  (\re{Zdec}). Hence, the right hand side 
 (\re{NR-est}) decays like $\sim t^{-2+\nu}$ where $-2+\nu<-3/2$
which is sufficient for the method of majorants 
(in integral inequalities (\re{duhest}) and (\re{bPX})).

}
\er
\setcounter{equation}{0}
\setcounter{equation}{0}
\section{Symplectic decomposition of the dynamics}
\la{Sdec}
Here we decompose the dynamics in two components: along the manifold
${\cal S}$ and in transversal directions. The equation (\re{lin})
is obtained without any assumption on $\si(t)$ in (\re{dec}).
We are going to choose $S(\si(t)):={\bf\Pi} Y(t)$, but then we need
to know  that
\be\la{YtO}
Y(t)\in \cO_\al(\cS),~~~~~t\in\R
\ee
with some $\cO_\al(\cS)$ defined in Lemma \re{skewpro}.
It is true for $t=0$ by our main assumption
(\re{close}) with sufficiently small $d_0>0$.
Then  $S(\si(0))={\bf\Pi} Y(0)$ and  $X(0)=Y(0)-S(\si(0))$
are well defined. We will prove below that (\re{YtO}) holds
with $\al=-\beta$ if $d_0$ is sufficiently small.
First, we choose $\ov v<1$ such that
\be\la{vsigmat}
|v(0)|\le \ov v
\ee
Denote by $r_{-\beta}(\ov v)$ the positive
number from Lemma \re{skewpro} iii) which corresponds to $\al=-\beta$.
Then $S(\si)+X\in \cO_{-\beta}(\cS)$ if $\si=(b,v)$ with $|v|<\ov v$ and
$ \Vert X\Vert_{E_{-\beta}}<r_{-\beta}(\ov v)$.
Therefore, $S(\si(t))={\bf\Pi}Y(t)$ and  $X(t)=Y(t)-S(\si(t))$
are well defined for $t\ge 0$ so small that
$\Vert X(t)\Vert_{E_{-\beta}} < r_{-\beta}(\ov v)$.
This is formalized by the standard definition of the ``exit time''.
First, we introduce the ``majorants''
\be\la{maj}
m_1(t):=\sup_{s\in[0,t]}(1+s)^{3/2}\Vert X(s)\Vert_{E_{-\beta}},~~~~~
m_2(t):=\sup_{s\in[0,t]}(1+s)^{1/2}\Vert\Psi(s)\Vert_{L^\infty}
\ee
Here $X=(X_1, X_2)=(\Psi, \Pi)$.
Let us denote by $\ve\in(0, r_{-\beta}(\ov v))$
a fixed  number which we will specify below.
\begin{definition} $t_{*}$ is the exit time
\be\la{t*}
t_*=\sup \{t\in[0,t_*): m_j(s)< \ve,\quad j=1,2,~~0\le s\le t\}
\ee
\end{definition}
Let us note that $m_j(0)<\ve$ if $d_0\ll 1$.
One of our main goals is to prove that $t_*=\infty$ if $d_0$ is
sufficiently small. This would follow if we show that
\be\la{Zt}
m_j(t)<\ve/2,~~~~~0\le t < t_*
\ee

\setcounter{equation}{0}
\section{Modulation equations }
\label{modsec}
In this section we present the modulation equations which allow
to construct the solutions $Y(t)$ of the equation (\re{Eq})
close at each time $t$ to a kink i.e. to one of the functions
described in Definition \ref {defsol} with time varying (``modulating'') parameters
$(b,v)=(b(t),v(t))$.
We look for a solution to (\ref{Eq}) in the form
$Y(t)=S(\si(t))+X(t)$
by setting $S(\si(t))={\bf\Pi} Y(t)$ which is equivalent to the
symplectic orthogonality condition of type (\re{commut}),
\be\la{ortZ}
X(t)\nmid{\cal T}_{S(\si(t))}{\cal S},\quad t<t_*
\ee
The projection ${\bf\Pi} Y(t)$ is well defined for $t<t_*$ by Lemma \re{skewpro} iii).
Now we derive the ``modulation equations'' for the parameters
$\si(t)=(b(t),v(t))$. For this purpose, let us write (\re{ortZ}) in the form
\be\la{orth}
\Om(X(t),\tau_j(t))=0,\,\,j=1,2
\ee
where the vectors $\tau_j(t)=\tau_j(\si(t))$ span the tangent space
$\cT_{S(\si(t))}{\cal S}$.
It would be convenient for us to use some other
parameters $(c,v)$ instead of $\si=(b,v)$, where
$c(t)= b(t)-\ds\int^t_0 v(\tau)d\tau$ and
\be\la{vw}
\dot c(t)=\dot b(t)-v(t)=w(t)-v(t)
\ee
\begin{lemma}\la{mod}
Let $Y(t)$ be a solution to the Cauchy problem (\re{Eq}), and (\re{dec}),
(\re{orth}) hold. Then the parameters $c(t)$ and $v(t)$ satisfy the equations
\be\la{cdot}
\dot c=\frac{\Om(\tau_1,\tau_2)\Om({\cal N},\tau_2)
+\Om(X,\pa_v\tau_1)\Om({\cal N},\tau_2)-\Om(X,\pa_v\tau_2)\Om({\cal N},\tau_1)}{D}
\ee
\be\la{vdot}
\dot v=\frac{-\Om(\tau_1,\tau_2)\Om({\cal N},\tau_1)
-\Om(X,\tau'_2)\Om({\cal N},\tau_1)-\Om(X,\tau'_1)\Om({\cal N},\tau_2)}{D}
\ee
where
$$
D=\Om^2(\tau_1,\tau_2)+\cO(\Vert X\Vert_{E_{-\beta}})
$$
\end{lemma}
\begin{proof} Differentiating  the orthogonality conditions
(\re{orth})  in $t$ we  obtain
\be\la{ortder}
0=\Omega(\dot X,\tau_j)+\Omega(X,\dot\tau_j)=\Omega(A_{v,w}X+T+{\cal N},\tau_j)+
\Omega(X,\dot\tau_j),\quad j=1,2
\ee
First, let us compute the principal (i.e. non-vanishing at $X=0$)
term $\Om(T,\tau_j)$. By (\re{Ttang}), 
\be\la{OmTtau}
\Omega(T,\tau_1)=-\dot v\Om(\tau_2,\tau_1)=\dot v\Om(\tau_1,\tau_2);\quad
\Omega(T,\tau_2)=-\dot c\Om(\tau_1,\tau_2)
\ee
Second, let us compute $\Omega(A_{v,w}X,\tau_j)$. The skew-symmetry (\re{com})
implies that $\Omega(A_{v,w}X,\tau_j)=-\Omega(X,A_{v,w}\tau_j)$.
Then   by (\re{Atan}) we have
\be\la{omaz1}
\Omega(A_{v,w}X,\tau_1)=\Omega(X,\dot c\tau'_1)
\ee
\be\la{omaz2}
\Omega(A_{v,w}X,\tau_2)=-\Omega(X,\dot c\tau'_2+\tau_1)=
-\Omega(X,\dot c\tau'_2)-\Omega(X,\tau_1)= -\Omega(X,\dot c\tau'_2)
\ee
since $\Om(X,\tau_1)=0$.

Finally, let us compute the last term $\Om(X,\dot\tau_j)$. For $j=1,2$ one has
$\dot\tau_j=\dot b\pa_b\tau_j+ \dot v\pa_v\tau_j= \dot v\pa_v\tau_j$
since the vectors $\tau_j$ do not depend on $b$ according to (\re{inb}).
Hence,
\be\la{Ztau}
\Omega(X,\dot\tau_j)= \Om(X,\dot v\pa_v\tau_j)
\ee
As the result, by (\re{OmTtau})-(\re{Ztau}),
 the equation (\re{ortder}) becomes
$$
0=\dot c\Om(X,\tau'_1)+\dot v\Big(\Om(\tau_1,\tau_2)+\Om(X,\pa_v\tau_1)\Big)
+\Om({\cal N},\tau_1),
$$
$$
0=-\dot c\Big(\Om(X,\tau'_2)+(\Om(\tau_1,\tau_2)\Big)
+\dot v\Om(X,\pa_v\tau_2)+\Om({\cal N},\tau_2)
$$
Since $\Om(\tau_1,\tau_2)\not =0$ by (\ref{symp}) then the determinant $D$
of the system does not vanish  for small
$\Vert X\Vert_{E_{-\beta}}$ and we obtain (\ref{cdot})-(\ref{vdot}).
\end{proof}
\begin{cor}\la{cv}
Formulas (\ref{cdot})-(\ref{vdot}) imply
\be\la{parameq}
|\dot c(t)|,~|\dot v(t)|\le C(\ov v)\Vert\Psi(t)\Vert_{L^2_{-\beta}}^2
\le C(\ov v)\Vert X(t)\Vert_{E_{-\beta}}^2, ~~~~~~~0\le t<t_*
\ee
\end{cor}
\setcounter{equation}{0}
\section{Decay for the transversal dynamics}
\la{Trans-dec}
In Section \re{sol-as} we will show that our main Theorem \re{main}
can be derived from the following time decay of the
transversal component $X(t)$:
\bp\la{pdec}
 Let all conditions of Theorem \re{main} hold. Then $t_*=\infty$, and
\be\la{Zdec}
\Vert X(t)\Vert_{E_{-\beta}}\le \ds\fr {C(\ov v,d_0)}{(1+|t|)^{3/2}},\quad
\Vert \Psi(t)\Vert_{L^\infty}\le \ds\fr {C(\ov v,d_0)}{(1+|t|)^{1/2}},
~~~~~t\ge0
\ee
\ep
We will derive (\re{Zdec}) in Sections \re{tr-dec} from our equation
(\re{lin}) for the transversal component $X(t)$.
This equation can be specified using Corollary \re{cv}.
Indeed, (\ref{Ttang})  implies that
\be\la{Tta}
\Vert T(t)\Vert_{E_{\beta}\cap W}\le C(\ov v)\Vert X\Vert_{E_{-\beta}}^2,
~~~~~~~~~0\le t<t_*
\ee
by (\re{parameq})  since $w-v=\dot c$.
Thus (\re{lin}) becomes the equation
\be\la{reduced}
\dot X(t)=A(t)X(t)+T(t)+{\cal N}_I(t)+{\cal N}_R(t), ~~~~~~~~~0\le t<t_*
\ee
where
$A(t)=A_{v(t),w(t)}$, $T(t)$ satisfies (\ref{Tta}), and
\be\la{Rem}
\left.
\ba{l}
\Vert {\cal N}_I(t)\Vert_{E_{\beta}\cap W}\le C(\ov v)
\Vert\Psi\Vert_{L^\infty}\Vert X\Vert_{E_{-\beta}}\\
\Vert {\cal N}_R\Vert_{E_{5/2+\nu}}\le C(\ov v)
(1+t)^{4+\nu}\Vert \Psi\Vert_{L^\infty}^{12},\quad 0<\nu<1/2\\
\Vert {\cal N}_R\Vert_{W}\le C(\ov v)\Vert \Psi\Vert_{L^\infty}^{10}
\ea
\right|~~~0\le t<t_*
\ee
by (\re{NI-est}), (\re{NR-est})-(\re{NR-est1}).
In all remaining part of our paper we will analyze mainly the
equation (\re{reduced}) to establish the decay (\re{Zdec}).
We are going to derive the decay using the bounds (\re{Tta}) and (\re{Rem}),
and the orthogonality condition  (\re{ortZ}).

Let us comment on two main difficulties in proving (\re{Zdec}).
The difficulties are common for the problems studied in \ci{BP}.
First, the linear part of the equation is non-autonomous,
hence we cannot apply directly the  methods of scattering theory.
Similarly to the approach of  \ci{BP}, we reduce the problem to
the analysis of the {\it frozen} linear equation,
\be\la{Avv}
\dot X(t)=A_1X(t), ~~t\in\R
\ee
where $A_1$ is the operator $A_{v_1}$ defined by (\re{AA})
with $v_1=v(t_1)$ for a fixed $t_1\in[0,t_*)$. Then we
estimate the error by the method of majorants.

Second, even for the frozen equation (\re{Avv}), the decay
of type  (\re{Zdec}) for all solutions does not hold without  the
orthogonality condition  of type (\re{ortZ}).
Namely, by  (\re{Atanform}) the equation
(\re{Avv}) admits the {\it secular solutions}
\be\la{secs}
X(t)=C_{1}\tau_1(v)+C_2[\tau_1(v)t+\tau_{2}(v)]
\ee
which arise also by differentiation of the soliton (\re{sosol})
in the parameters $q$ and $v$ in the moving coordinate $y=x-v_1t$.
Hence, we have to take into account the orthogonality condition
(\re{ortZ}) in order to avoid the secular solutions.
For this purpose we will apply the corresponding symplectic orthogonal
projection which kills the ``runaway solutions''  (\re{secs}).
\br
{\rm
The solution (\re{secs}) lies in the tangent space  $\cT_{S(\si_1)}{\cal S}$
with $\si_1=(b_1,v_1)$ (for an arbitrary $b_1\in\R$)
that suggests an unstable character of the nonlinear dynamics
{\it along the solitary manifold} (cf. Remark \re{rT} ii)).
}
\er
\bd
Denote by ${\cal X}_v=\bP^c_v E$ the space symplectic
orthogonal to $\cT_{S(\si)}{\cal S}$ with
$\si=(b,v)$ (for an arbitrary $b\in\R$).
\ed
Now we have the symplectic orthogonal decomposition
\be\la{sod}
E=\cT_{S(\si)}{\cal S}+{\cal X}_v,~~~~~~~\si=(b,v)
\ee
and the symplectic orthogonality  (\re{ortZ})
can be written in the following equivalent forms,
\be\la{PZ}
\bP_{v(t)}^dX(t)=0,~~~~\bP^c_{v(t)}X(t)=X(t),~~~~~~~0\le t<t_*
\ee

\br\la{rZ}
{\rm
The tangent space $\cT_{S(\si)}{\cal S}$ is invariant under the operator
$A_{v}$ by (\re{Atanform}), hence the space  ${\cal X}_v$ is also
invariant by (\re{com}): $A_{v}X\in {\cal X}_v$
on a dense domain of $X\in {\cal X}_v$.
}
\er

\setcounter{equation}{0}
\section{Frozen Form of Transversal Dynamics}
\la{Froz}
Now  let us fix an arbitrary $t_1\in [0,t_*)$, and
rewrite the equation (\re{reduced}) in a ``frozen form''
\be\la{froz}
\dot X(t)=A_1X(t)+(A(t)-A_1)X(t)+T(t)+{\cal N}_I(t)+{\cal N}_R(t),\quad 0\le t<t_*
\ee
where $A_1=A_{v(t_1),v(t_1)}$ and
$$
A(t)-A_1=\left(
\ba{cc}
(w(t)\!-\!v(t_1))\nabla &                 0 \\
0                       & (w(t)\!-\!v(t_1))\nabla
\ea
\right)
$$
The next trick is important since it allows us to kill the ``bad terms''
$(w(t)-v(t_1))\nabla$ in the operator $A(t)-A_1$.
\begin{definition}\la{d71}
Let us change the  variables $(y,t)\mapsto (y_1,t)=(y+d_1(t),t)$
where
\be\la{dd1}
d_1(t):=\int_{t_1}^t(w(s)-v(t_1))ds, ~~~~0\le t\le t_1
\ee
\end{definition}
Next define
\be\la{Z1}
\ti X(t)=(\Psi(y_1-d_1(t),t),\Pi(y_1-d_1(t),t))
\ee
Then we obtain the final form of the
``frozen equation'' for the transversal dynamics
\be\la{redy1}
\dot{\ti X}(t)=A_1\ti X(t)+\ti T(t)+\ti {\cal N}_I(t)+\ti {\cal N}_R(t) ,
\quad 0\le t\le t_1
\ee
where $\ti T(t)$, $\ti {\cal N}_I(t)$ and $\ti {\cal N}_R(t)$ are
$T(t)$, ${\cal N}_I(t)$ and ${\cal N}_R(t)$ expressed in terms of $y_1=y+d_1(t)$.
At the end of this section, we will derive appropriate bounds for the
``remainder terms''  in (\re{redy1}).
Let us recall the following well-known inequality: for any $\al\in\R$
\be\la{pitre}
(1+|y+x|)^{\alpha}\le(1+|y|)^{\alpha}(1+|x|)^{|\alpha|},
\,\,\,~~~~~~x,y\in\R
\ee
\begin{lemma}\la{dest}
For $f\in L^2_{\alpha}$ with any $\alpha\in\R$
the following bound holds:
\be\la{shiftest}
\Vert f(y_1-d_1)\Vert_{L^2_{\al}}\le
\Vert f\Vert_{L^2_{\al}}(1+|d_1|)^{|\al|}\quad d_1\in\R
\ee
\end{lemma}
\begin{proof}
The bound (\re{shiftest}) follows from (\re{pitre}) since
$$
\Vert f(y_1-d_1)\Vert^2_{L^2_\al}=\int |f(y_1-d_1)|^2
(1+|y_1|)^{2\alpha}dy_1=\int |f(y)|^2(1+|y+d_1|)^{2\alpha}dy\le
$$
$$
\int|f(y)|^2(1+|y|)^{2\alpha}(1+|d_1|)^{2|\alpha|}dy\le
(1+|d_1|)^{2|\alpha|}\Vert f\Vert^2_{L^2_\alpha}
$$
\end{proof}
\begin{cor}\la{cor1}
The following  bounds hold for $0\le t\le t_1$ by (\re{Tta}) and (\re{Rem}):
\be\la{N1est}
\left.\ba{rl}
&\Vert \ti T(t)\Vert_{E_{\beta}}\le C(\ov v)(1+|d_1(t)|)^{\beta}
\Vert X\Vert_{E_{-\beta}}^2,\quad
\Vert \ti T(t)\Vert_{ W}\le C(\ov v)\Vert X\Vert_{E_{-\beta}}^2\\
&\Vert\ti {\cal N}_I(t)\Vert_{E_{\beta}}\le C(\ov v)
(1+|d_1(t)|)^{\beta}\Vert\Psi\Vert_{L^\infty}\Vert X\Vert_{E_{-\beta}},\quad
\Vert\ti {\cal N}_I(t)\Vert_{ W}\le C(\ov v)
\Vert\Psi\Vert_{L^\infty}\Vert X\Vert_{E_{-\beta}}\\
&\Vert\ti {\cal N}_R\Vert_{E_{5/2+\nu}}\le C(\ov v)(1+|d_1(t)|)^{5/2+\nu}
(1+t)^{4+\nu}\Vert \Psi\Vert_{L^\infty}^{12},\quad 0<\nu<1/2\\
&\Vert\ti {\cal N}_R\Vert_{W}\le C(\ov v)\Vert \Psi\Vert_{L^\infty}^{10}
\ea\right|
\ee
\end{cor}

\setcounter{equation}{0}
\section{Integral inequality}
\la{int-in}
The equation (\re{redy1}) can be written in the integral form:
\be\la{Z1duh}
\ti X(t)=e^{A_1t}\ti X(0)+\int_0^te^{A_1(t-s)}[\ti T(s)+\ti {\cal N}_I(s)+\ti {\cal N}_R(s)]ds,
\quad 0\le t\le t_1
\ee
We apply the symplectic orthogonal
projection $\bP_1^c:=\bP_{v(t_1)}^c$ to both sides, and get
$$
\bP_1^c\ti X(t)=e^{A_1t}\bP_1^c\ti X(0)+\int_0^te^{A_1(t-s)}\bP_1^c~
[\ti T(s)+\ti {\cal N}_I(s)+\ti {\cal N}_R(s)]~ds
$$
We have used here that $\bP_1^c$ commutes with the group $e^{A_1t}$
since the space ${\cal X}_1:=\bP_1^c E$ is invariant with respect to
$e^{A_1t}$ by Remark \re{rZ}. Applying (\re{t-as1}) we obtain that
$$
\Vert \bP_1^c\ti X(t)\Vert_{E_{-\beta}}
\le\fr{C\Vert\ti X(0)\Vert_{E_\beta}}{(1+t)^{3/2}}
+C\int_0^t \fr{\Vert\ti T(s)+\ti {\cal N}_I(s)+\ti {\cal N}_R(s)
\Vert_{E_\beta}}{(1+|t-s|)^{3/2}}ds
$$
Hence, for $5/2<\beta<3$ and $0\le t\le t_1$
the bounds (\re{N1est}) imply
\beqn\la{duhest}
&&\Vert \bP_1^c\ti X(t)\Vert_{E_{-\beta}}\le
\fr{C(\ov d_1(0))}{(1+t)^{3/2}}\Vert X(0)\Vert_{E_\beta}\\
\nonumber
&&+C(\ov d_1(t))\int_0^t\fr{\Vert X(s)\Vert_{E_{-\beta}}^2
+\Vert\Psi(s)\Vert_{L^\infty}\Vert X(s)\Vert_{E_{-\beta}}
+(1+s)^{3/2+\beta}\Vert \Psi(s)\Vert_{L^\infty}^{12}}{(1+|t-s|)^{3/2}}ds
\eeqn
where $\ov d_1(t):=\sup_{0\le s\le t} |d_1(s)| $.
Similarly, (\re{t-as3}) and (\re{N1est}) imply
\bigskip
\beqn\nonumber
&&\Vert(\bP_1^c\ti X(t))_1\Vert_{L^\infty}
\le\ds\fr{C\Vert \ti X(0)\Vert_{E_\beta\cap W}}{(1+t)^{1/2}}
+C\int_0^t\fr{\Vert\ti T(s)+\ti {\cal N}_I(s)
+\ti {\cal N}_R(s)\Vert_{E_\beta\cap W}}{(1+|t-s|)^{1/2}}ds\\
\la{bPX}
&&\le\ds\fr{C(\ov d_1(0))}{(1+t)^{1/2}}\Vert X(0)\Vert_{E_\beta\cap W}\\
\nonumber
&&+C(\ov d_1(t))\int_0^t\fr{\Vert X(s)\Vert_{E_{-\beta}}^2+
\Vert\Psi(s)\Vert_{L^\infty}\Vert X(s)\Vert_{E_{-\beta}}
+(1+s)^{3/2+\beta}\Vert \Psi(s)\Vert_{L^\infty}^{12}
+\Vert \Psi(s)\Vert_{L^\infty}^{10}}{(1+|t\!-\!s|)^{1/2}}ds
\eeqn
\begin{lemma}\la{d1t} For $t_1<t_*$ we have
\be\la{dtt}
|d_1(t)|\le C\ve^2,\quad 0\le t \le t_1
\ee
\end{lemma}
\begin{proof}
To estimate $d_{1}(t)$, we note that
\be\la{wen}
w(s)-v(t_{1})=w(s)-v(s)+v(s)-v(t_{1})=
\dot c(s)+\int_s^{t_1}\dot v(\tau)d\tau
\ee
by (\re{vw}).
Hence, the definitions (\ref {dd1}), (\re{maj}), and
corollary \ref{cv}  imply that
\begin{eqnarray}
\!\!\!\!\!\! \!\!\! \!\!\!
 |d_{1}(t)|\!\!\! &=&\!\!\!|\int_{t_{1}}^{t}(w(s)-v(t_{1}))ds|\leq
\int_{t}^{t_{1}}\left( |\dot{c}(s)|+\int_{s}^{t_{1}}|\dot{v}(\tau )|d\tau
\right)ds   \nonumber  \label{d1est} \\
&&  \nonumber \\
\!\!\! &\leq &\!\!\!Cm_1^2(t_1)\int_{t}^{t_{1}}
\left( \frac{1}{(1+s)^{3}}%
+\int_{s}^{t_{1}}\frac{d\tau }{(1+\tau )^{3}}\right) ds\leq
Cm_1^2(t_1)\leq C\varepsilon ^{2},~~~~0\leq t\le t_{1}
\end{eqnarray}
\end{proof}
Now (\re{duhest}) and (\re{bPX}) imply that for $t_1<t_*$
and $0\le t\le t_1$
\beqn\la{duhestri}
&&\!\!\!\!\!\!\!\!\Vert \bP_1^c\ti X(t)\Vert_{E_{-\beta}}
\le\fr{C\Vert X(0)\Vert_{E_\beta}}{(1+t)^{3/2}}\\
\nonumber
&&\!\!\!\!\!\!\!\!+C\int_0^t\!\fr{\Vert X(s)\Vert_{E_{-\beta}}^2
+\Vert\Psi(s)\Vert_{L^\infty}\Vert X(s)\Vert_{E_{-\beta}}
+(1+s)^{3/2+\beta}\Vert \Psi(s)\Vert_{L^\infty}^{12}}{(1+|t-s|)^{3/2}}ds
\eeqn
\beqn\la{bPX1}
&&\!\!\!\!\!\!\!\!\Vert(\bP_1^c\ti X(t))_1\Vert_{L^\infty}
\le\fr{C\Vert X(0)\Vert_{E_\beta\cap W}}{(1+t)^{1/2}}\\
\nonumber
&&\!\!\!\!\!\!\!\!+C\int_0^t\!\fr{\Vert X(s)\Vert_{E_{-\beta}}^2
+\Vert\Psi(s)\Vert_{L^\infty}\Vert X(s)\Vert_{E_{-\beta}}
+(1+s)^{3/2+\beta}\Vert \Psi(s)\Vert_{L^\infty}^{12}
+\Vert \Psi(s)\Vert_{L^\infty}^{10}}{(1+|t-s|)^{1/2}}ds
\eeqn
\setcounter{equation}{0}
\section{Symplectic orthogonality}
\la{symp-ort}
Finally, we are going to change $\bP_1^c\ti X(t)$ by $X(t)$ in the
left hand side of (\re{duhestri}) and (\re{bPX1}).
We will prove that it is possible since  $d_0\ll 1$ in (\re{close}).
\begin{lemma}\la{Z1P1Z1}
For sufficiently small $\ve>0$, we have for $t_1<t_*$
$$
\Vert X(t)\Vert_{E_{-\beta}}\le C\Vert \bP_1^c\ti X(t)\Vert_{E_{-\beta}},
~~~~~~~~0\le t \le t_1
$$
$$
\Vert \Psi(t)\Vert_{L^\infty}\le 2\Vert (\bP_1^c\ti X(t))_1\Vert_{L^\infty},
~~~~~~~~0\le t \le t_1
$$
where  the constant $C$ does not depend on  $t_1$.
\end{lemma}
\begin{proof}
The proof is based on the symplectic orthogonality (\re{PZ}), i.e.
\be\la{PZ1}
\bP_{v(t)}^dX(t)=0,~~~~t\in[0,t_1]
\ee
and on the fact that all the spaces ${\cal X}(t):=\bP_{v(t)}^cE$ are almost
parallel for all $t$.

Namely, we first note that
$\Vert \Psi(t)\Vert_{L^\infty}=\Vert\ti\Psi(t)\Vert_{L^\infty}$, and
$\Vert X(t)\Vert_{E_{-\beta}}\le C\Vert\ti X(t)\Vert_{E_{-\beta}}$
by Lemma \re{dest}, since $|d_1(t)|\le\const$ for
$t\le t_1<t_*$ by (\ref{dtt}). Therefore, it suffices to prove that
\be\la{Z1P1ests}
\Vert\ti\Psi(t)\Vert_{L^\infty}\le 2\Vert(\bP_1^c\ti X(t))_1\Vert_{L^\infty},
\quad\Vert\ti X(t)\Vert_{E_{-\beta}}
\le 2\Vert \bP_1^c\ti X(t)\Vert_{E_{-\beta}},\quad 0\le t\le t_1
\ee
This estimate will follow from
\be\la{Z1P1estf}
\Vert(\bP_{1}^d\ti X(t))_1\Vert_{L^\infty}
\le\fr12\Vert \ti \Psi(t)\Vert_{L^\infty},\quad
\Vert\bP_{1}^d\ti X(t)\Vert_{E_{-\beta}}
\le\fr12\Vert \ti X(t)\Vert_{E_{-\beta}},
\,\,\,0\le t\le t_1
\ee
since $\bP_1^c\ti X(t)=\ti X(t)-\bP_{1}^d\ti X(t)$.
To prove (\re{Z1P1estf}), we write (\re{PZ1}) as
\be\la{PZ0r}
\ti\bP_{v(t)}^d\ti X(t)=0,~~~~t\in[0,t_1]
\ee
where $\ti\bP_{v(t)}^d\ti X(t)$ is $\bP_{v(t)}^dX(t)$ expressed in
terms of the variable $y_1=y+d_1(t)$.
Hence, (\re{Z1P1estf}) follows from  (\re{PZ0r}) if
the difference
 $\bP_{1}^d-\ti\bP_{v(t)}^d$ is small uniformly in $t$,
i.e.
\be\la{difs}
\Vert\bP_{1}^d-\ti\bP_{v(t)}^d\Vert<1/2,~~~~~~~0\le t\le t_1
\ee
It remains to justify (\ref{difs}) for small enough $\varepsilon >0.$
In order to prove the bound (\re{difs}), we will need the formula
(\ref{Piv}) and the following relation which follows from (\ref{Piv}):
\begin{equation} \label{P1}
\ti\bP_{v(t)}^d\ti X(t)=\sum p_{jl}(v(t))\ti\tau _{j}(v(t))
\Omega (\ti\tau _{l}(v(t)),\ti X(t))
\end{equation}
where $\ti\tau_{j}(v(t))$ are the vectors $\tau _{j}(v(t))$ expressed in the
variables $y_{1}$. In detail (cf. (\ref{inb})),
\begin{equation}\label{inb*}
\left.
\begin{array}{rcl}
\ti\tau _{1}(v) & := & (-\psi'_{v}(y_{1}-d_{1}(t)),
-\pi'_{v}(y_{1}-d_{1}(t))) \\
\ti\tau _{2}(v) & := & (\partial _{v}\psi _{v}(y_{1}-d_{1}(t)),\partial
_{v}\pi _{v}(y_{1}-d_{1}(t)))
\end{array}
\right.
\end{equation}
where $v=v(t)$. Since
$\tau'_{j}$ are smooth and rapidly decaying at infinity
functions, then Lemma \ref{d1t} implies
\begin{equation} \label{011}
\Vert \ti\tau _{j}(v(t))-\tau _{j}(v(t))\Vert_{E_\beta}\le
C\ve^2,~~~~0\leq t\leq t_{1},\quad j=1,2
\end{equation}
Furthermore,
$$
\tau _{j}(v(t))-\tau _{j}(v(t_{1}))=\int_{t}^{t_{1}}\dot{v}(s)
\pa_{v}\tau _{j}(v(s))ds,
$$
and therefore
\begin{equation}
\Vert \tau _{j}(v(t))-\tau _{j}(v(t_{1}))\Vert_{E_\beta}\leq
C\int_{t}^{t_{1}}|\dot{v}(s)|ds,~~~~0\leq t\leq t_{1}.  \label{012}
\end{equation}
\begin{equation}
|p_{jl}(v(t))-p_{jl}(v(t_{1}))|
= |\int_{t}^{t_{1}}\dot{v}(s)\pa_{v}p_{jl}(v(s))ds|\leq
C\int_{t}^{t_{1}}|\dot{v}(s)|ds,~~~~0\leq t\leq t_{1}, \label{013}
\end{equation}
since $|\pa_{v}p_{jl}(v(s))|$ is uniformly bounded by (\re{vsigmat}).
Further,
\begin{equation}
\int_{t}^{t_{1}}|\dot{v}(s)|ds\leq Cm_1^2(t_1)
\int_{t}^{t_{1}}\frac{ds}{(1+s)^{3}}\leq C\varepsilon ^{2},
~~~~0\leq t\le t_{1}.
\label{tvjest}
\end{equation}
Hence, the bounds (\ref{difs}) will follow from (\ref{Piv}),
(\ref{P1}) and (\ref{011})-(\ref{013}) if we choose $\ve >0$ small enough.
The proof is completed.
\end{proof}

\setcounter{equation}{0}
\section{Decay of transversal component}
\la{tr-dec}
Here we prove Proposition \re{pdec}.
\smallskip\\
{\it Step i)} We fix $\ve>0$ and $t_*=t_*(\ve)$
for which Lemma \re{Z1P1Z1} holds.
Then the bounds of type (\re{duhestri}) and (\re{bPX1}) holds with
$\Vert\bP_1^d\ti X(t)\Vert_{E_{-\beta}}$
and  $\Vert(\bP_1^d\ti X(t))_1\Vert_{L^\infty}$
in the left hand sides replaced by
$\Vert X(t)\Vert_{E_{-\beta}}$ and $\Vert\Psi(t)\Vert_{L^\infty}$~:
\be\la{duhestrih}
\Vert X(t)\Vert_{-\beta}\le\fr{C\Vert X(0)\Vert_{E_\beta}}{(1+t)^{3/2}}
+C\int_0^t\fr{\Vert X(s)\Vert_{E_{-\beta}}^2
+\Vert\Psi(s)\Vert_{L^\infty}\Vert X(s)\Vert_{E_{-\beta}}
+(1+s)^{3/2+\beta}\Vert \Psi(s)\Vert_{L^\infty}^{12}}{(1+|t-s|)^{3/2}}ds
\ee
\beqn\la{bPX2}
&&\!\!\!\!\!\!\!\!\!\!\!\!\!\!\!\!\!\!\!\!\!\!\!\!\!\!\!\!\!\!\!\!\!\!
\Vert \Psi(t)\Vert_{L^\infty}
\le\ds\fr{C\Vert X(0)\Vert_{E_\beta\cap W}}{(1+t)^{1/2}}\\
\nonumber
&&\!\!\!\!\!\!\!\!\!\!\!\!\!\!\!\!\!\!\!\!\!\!\!\!\!\!\!\!\!\!\!\!\!\!
+C\int_0^t\fr{\Vert X(s)\Vert_{E_{-\beta}}^2
+\Vert\Psi(s)\Vert_{L^\infty}\Vert X(s)\Vert_{E_{-\beta}}
+(1+s)^{3/2+\beta}\Vert \Psi(s)\Vert_{L^\infty}^{12}
+\Vert \Psi(s)\Vert_{L^\infty}^{10}}{(1+|t\!-\!s|)^{1/2}}ds
\eeqn
for $0\le t\le t_1$ and $t_1<t_*$.
This implies an integral inequality for the majorants
$m_1$ and $m_2$.
Namely, multiplying both sides of (\re{duhestrih})
by $(1+t)^{3/2}$, and taking the supremum in $t\in[0,t_1]$, we obtain
$$
m_1(t_1)\le C\Vert X(0)\Vert_{E_\beta}+
C\!\!\sup_{t\in[0,t_1]}\ds
\int_0^t\!\!\fr{(1+t)^{3/2}ds}{(1+|t\!-\!s|)^{3/2}}
\left[\fr{m_1^2(s)}{(1+s)^{3}}+\fr{m_1(s)m_2(s)}{(1+s)^{2}}
+\fr{m_2^{12}(s)(1+s)^{3/2+\beta}}{(1\!+s)^{6}}\right]
$$
for $t_1< t_*$.
Taking into account that $m(t)$ is a monotone
increasing function, we get
\be\la{mest}
m_1(t_1)\le C\Vert X(0)\Vert_{E_\beta}
+C[m_1^2(t_1)+m_1(t_1)m_2(t_1)+m_2^{12}(t_1)]I_1(t_1),\quad t_1<t_*
\ee
where
$$
I_1(t_1)=\sup_{t\in[0,t_1]}
\int_0^{t}\fr{(1+t)^{3/2}}{(1+|t-s|)^{3/2}}\fr{ds}{(1+s)^{9/2-\beta}}
\le \ov I_1<\infty,~~~~t_1\ge0,\quad 5/2<\beta<3
$$
Therefore, (\re{mest}) becomes
\be\la{m1est}
m_1(t_1)\le C\Vert X(0)\Vert_{E_\beta}
+C\ov I_1[m_1^2(t)+m_1(t_1)m_2(t_1)+m_2^{12}(t_1)],~~~~ t_1<t_*
\ee
Similarly,  multiplying both sides of (\re{bPX2})by $(1+t)^{1/2}$,
and taking the supremum in $t\in[0,t_1]$, we get
\be\la{mest1}
m_2(t_1)\le C\Vert X(0)\Vert_{E_\beta\cap W}
+C[m_1^2(t_1)+m_1(t_1)m_2(t_1)+m_2^{12}(t_1)
+m_2^{10}(t_1)]I_2(t_1),\quad t_1<t_*
\ee
where
$$
I_2(t_1)=\sup_{t\in[0,t_1]}
\int_0^{t}\fr{(1+t)^{1/2}}{(1+|t-s|)^{1/2}}\fr{ds}{(1+s)^{9/2-\beta}}
\le \ov I_2<\infty,~~~~t_1\ge0,\quad 5/2<\beta<3
$$
Therefore, (\re{mest1}) becomes
\be\la{m1est1}
m_2(t_1)\le C\Vert X(0)\Vert_{E_\beta\cap W}
+C\ov I_2[m_1^2(t_1)+m_1(t_1)m_2(t_1)+m_2^{12}(t_1)+m_2^{10}(t_1)]
~~~~ t_1<t_*
\ee
Inequalities (\re{m1est}) and (\re{m1est1}) imply that $m_1(t_1)$
and $m_2(t_1)$ are bounded for $t_1<t_*'$, and moreover,
\be\la{m2est}
m_1(t_1),~m_2(t_1)\le C\Vert X(0)\Vert_{E_\beta\cap W},\quad t_1<t_*
\ee
since $m_1(0)=\Vert X(0)\Vert_{E_{-\beta}}$ and
$m_2(0)=\Vert \Psi(0)\Vert_{L^\infty}$
are sufficiently small by (\re{close}).
\\
{\it Step ii)} The constant $C$  in the estimate
(\re{m2est}) does not depend on $t_*$ by Lemma \re{Z1P1Z1}.
We choose $d_0$ in (\re{close}) so small that
$\Vert X(0)\Vert_{E_{\beta}\cap W}<\ve/(2C)$. It is possible due to (\re{close}).
Finally, this implies that
$t_*=\infty$, and (\re{m2est}) holds for all $t_1>0$ if $d_0$ is small enough.

\section{Soliton asymptotics}
\setcounter{equation}{0}
\la{sol-as}
Here we prove our main Theorem \re{main} under the assumption
that the decay (\re{Zdec}) holds.
The  estimates (\re{parameq}) and (\re{Zdec}) imply that
\be\la{bv}
|\dot c(t)|+|\dot v(t)|\le \ds\fr {C_1(\ov v,d_0)}{(1+t)^{3}},
~~~~~~t\ge0
\ee
Therefore, $c(t)=c_+ +\cO(t^{-2})$ and $v(t)=v_+ +\cO(t^{-2})$, $t\to\infty$.
Similarly,
\be\la{bt}
b(t)=c(t)+\ds\int_0^tv(s)ds=v_+t+q_++\al(t),\quad\al(t)=\cO(t^{-1})
\ee
We have obtained the solution  $Y(x,t)=(\psi(x,t),\pi(x,t))$ to (\re{eq})
in the form
\be\la{spl}
Y(x,t)=Y_{v(t)}(x-b(t),t)+X(x-b(t),t)
\ee
where we define now $v(t)=\dot b(t)=v_++\dot\al(t)$. Since
$$
\Vert Y_{v(t)}(x-b(t),t)-Y_{v_+}(x-v_+t-q_+,t)\Vert_E=\cO(t^{-1}),
$$
it remains to extract the dispersive wave $W_0(t)\Phi_+$ from the term
$X(x-b(t),t)$. Substituting (\re{spl}) into (\re{eq}) we obtain
by (\re{stfch})
the inhomogeneous Klein-Gordon equation for the  $X(x-b(t),t)$:
\be\la{sys}
\dot X(y,t)=A^0_v X(y,t)+ R(y,t),\quad 0\le t\le\infty
\ee
where $y=x-b(t)$, and
$$
A^0_v=\left(\ba{cc}
v\nabla        & 1      \\
\Delta-m^2     & v\nabla
\ea
\right),\quad
R(t)=\left(\ba{c}
\dot v\pa_v\psi_v\\
\dot v\pa_v\pi_v+F(\Psi+\psi_v)-F(\psi_v)+m^2\Psi
\ea\right)
$$
Now we change the variable $y\mapsto y_1=y+\al(t)+q_+$.
Then we obtain the ``frozen'' equation
\be\la{reds}
\dot{\ti X}(t)=A_{+}\ti X(t)+\ti R(t),
\quad 0\le t\le\infty,
\ee
where $\ti X(t)$ and $\ti R(t)$ are $X(t)$ and $R(t)$  of $y=y_1-\al(t)-q_+$,
and
\be\la{A+}
A_+=\left(
\ba{cc}
v_+\nabla      &   1 \\
\Delta-m^2 &   v_+\nabla
\ea
\right)
\ee
Equation (\re{reds}) implies
\be\la{eqacc}
{\ti X}(t)=W_+(t){\ti X}(0)+\int_0^tW_+(t-s){\ti R}(s)ds\\
\ee
where $W_+(t)=e^{A_{+}t}$ is the integral operator with integral kernel
$$
W_+(y_1-z,t)=W_0(y_1-z+v_+t,t)=W_0(x-z,t)
$$
since by (\ref{bt})
$$
y_1+v_+t=y+\al(t)+q_++v_+t=x-b(t)+\al(t)+q_++v_+t=x
$$
Hence, equation (\ref{eqacc}) implies
\be\la{eqacc1}
X(x-b(t),t)=W_0(t)\ti X(0)+\int_0^tW_0(t-s)\ti R(s)ds
\ee
Let us rewrite (\re{eqacc1}) as
$$
X(x-b(t),t)=W_0(t)\Big(\ti X(0)+
\!\!\int\limits_0^{\infty}\!W_0(-s)\ti R(s)ds\Big)
-\!\!\int\limits^{\infty}_t\!W_0(t\!-\!s)\ti R(s)ds
=W_0(t)\Phi_{+}+r_{+}(t)
$$
To establish the asymptotics (\re{S}), it suffices to prove that
\be\la{TD}
\Phi_+=\ti X(0)+
\!\!\int\limits_0^{\infty}\!W_0(-s)\ti R(s)ds\in E\quad
{\rm and}\quad \Vert r_+(t)\Vert_{E}={\cal O}(t^{-1/2})
\ee
Assumption (\re{close}) implies that $\ti X(0)\in E$.
Let us split $\ti R(s)$ as the sum
$$
\ti R(s)=\left(\ba{c}
\dot v\pa_v\ti\psi_v\\
\dot v\pa_v\ti\pi_v
\ea\right)+\left(\ba{c}
0\\
F(\ti\Psi+\ti\psi_v)-F(\ti\psi_v)+m^2\ti\Psi
\ea\right)=\ti R'(s)+\ti R''(s)
$$
By (\re{bv}), we obtain
\be\la{Rs1}
\Vert\ti R'(s)\Vert_{E}=\cO(s^{-3})
\ee
Let us consider $\ti R''=(0,\ti R''_2)$. We have
$$
\ti R''_2=F(\ti\Psi+\ti\psi_v)-F(\ti\psi_v)+m^2\ti\Psi
=(F'(\ti\psi_v)+m^2)\ti\Psi+\ti N(v,\ti\Psi)
=-\ti V_v\ti\Psi+\ti N(v,\ti\Psi)
$$
By (\ref{V-decay}) and (\ref{Zdec}), we obtain
\be\la{Rs3}
\Vert\ti V_v\ti\Psi(t)\Vert_{L^2}\le C\Vert\ti\Psi(t)\Vert_{L^2_{-\beta}}
\le \ds C(\ov v,d_0)(1+|t|)^{-3/2}
\ee
since $|q_++\al(t)|\le C$. Finally, (\ref{Zdec}), (\re{Rem}), and (\re{shiftest}) imply
\be\la{Rs4}
\Vert\ti N(v,\ti\Psi(t))\Vert_{L^2}\le \ds C(\ov v,d_0)(1+|t|)^{-3/2}
\ee
Hence, (\ref{Rs3})-(\ref{Rs4}) imply
\be\la{Rs2}
\Vert\ti R''(s)\Vert_{E}=\cO(s^{-3/2})
\ee
and (\re{TD}) follows by (\ref{Rs1}) and (\ref{Rs2}).
\setcounter{section}{0}
\setcounter{equation}{0}
\protect\renewcommand{\thesection}{\Alph{section}}
\protect\renewcommand{\theequation}{\thesection. \arabic{equation}}
\protect\renewcommand{\thesubsection}{\thesection. \arabic{subsection}}
\protect\renewcommand{\thetheorem}{\Alph{section}.\arabic{theorem}}
\section{Virial type estimates}
\setcounter{equation}{0}
\label{wes}
Here we prove the weighted estimate (\re{jv1}).
Let us recall that we split the solution
$Y(t)=(\psi(\cdot,t),\pi(\cdot,t))=S(\si(t))+X(t)$,
 and denote $X(t)=(\Psi(t),\Pi(t))$,
$(\Psi_0,\Pi_0):=(\Psi(0),\Pi(0))$.
Our basic  condition (\ref{close}) implies that for some $\nu>0$
\be\la{inn}
\Vert X_0\Vert_{E_{5/2+\nu}}\le d_{0}<\infty
\ee
\begin{pro}\la{ee2}
Let the potential $U$ satisfy conditions {\bf U1}, and
$\Psi_0$ satisfy (\re{inn}).
 Then the bounds hold
\be\la{solt}
\Vert\Psi(t)\Vert_{L^2_{5/2+\nu}}\le C(\ov v,d_{0})(1+t)^{4+\nu},
~~~~~~~~t>0
\ee
\end{pro}
We will deduce the proposition from the following two lemmas.
The first lemma is well known.
Denote
$$
e(x,t)=\ds\fr{|\pi(x,t)|^2}2+\ds\fr{|\psi'(x,t)|^2}2+ U(\psi(x,t)).
$$
\begin{lemma}\la{ee}
For the solution $\psi(x,t)$ of Klein-Gordon equation (\ref{e})
the local energy estimate holds
\be\la{enes}
 \int\limits_{a_1}^{a_2}e(x,t)~dx
 \le\int\limits_{a_1-t}^{a_2+t}e(x,0)~dx,\quad a_1<a_2,\quad t>0.
\ee
\end{lemma}
\begin{proof}
The estimate follows by standard arguments: multiplication
of the equation (\re{e})
by $\dot\psi(x,t)$ and integration over the trapezium $ABCD$, where
$A=(a_1-t,0)$, $B=(a_1,t)$, $C=(a_2,t)$, $D=(a_2+t,0)$.
Then (\ref{enes}) is obtained after partial integration
using that $U(\psi)\ge 0$.
\end{proof}
\begin{lemma}\la{ee1}
For any $\si\ge 0$
\be\la{wees}
\int(1+|x-b|^{\si})e(x,t)dx \le C(\si)(1+t+|b|)^{\si+1}\int(1+|x|^{\si})e(x,0)dx.
\ee
\end{lemma}
\begin{proof}
By (\ref{enes})
$$
\int (1+|y|^{\si})\Big(\int\limits_{y+b-1}^{y+b}e(x,t)dx\Big)dy
\le\int (1+|y|^{\si})\Big(\int\limits_{y+b-1-t}^{y+b+t}e(x,0)dx\Big)dy
$$
Hence,
\be\la{wees1}
\int e(x,t)\Big(\int\limits_{x-b}^{x-b+1}(1+|y|^{\si})dy\Big)dx
\le\int e(x,0)\Big(\int\limits_{x-b-t}^{x-b+1+t}(1+|y|^{\si})dy\Big)dx
\ee
Obviously,
\be\la{y7}
\int\limits_{x-b}^{x-b+1}(1+|y|^\si)dy\ge c(\si)(1+|x-b|^\si)
\ee
with some  $c(\sigma)>0$. On the other hand,
\be\la{yt7}
\int\limits_{x-b-t}^{x-b+1+t}(1+|y|^\si)dy\le (2t+1)(1+t+|b|+|x|)^{\si}
\le C(1+t+|b|)^{\si+1}(1+|x|^{\si})
\ee
since $\si\ge 0$.
Finally, (\ref{wees1})-(\ref{yt7}) imply (\ref{wees}).
\end{proof}
~\\
{\bf Proof of Proposition \ref{ee2}}
First, we verify that
\be\la{U-int}
U_0=\int(1+|x|^{5+2\nu})U(\psi_0(x))dx
<C(d_0),\quad\psi_0(x)=\psi(x,0)
\ee
Indeed, $\psi_0(x)=\psi_{v_0}(x-q_0)+\Psi_0(x)$ is bounded since
$\Psi_0\in H^1(\R)$.
Hence
{\bf U1} implies that
$$
|U(\psi_0(x))|\le C(d_0)(\psi_0(x)\pm a)^2\le C(d_0)\Big((\psi_{v_0}(x-q_0)\pm a)^2
+\Psi_0^2(x)\Big)
$$
and then (\ref{U-int}) follows by (\ref{sol}), (\ref{s-decay}) and (\re{inn}).
Further, we have
\be\la{en}
\Vert\Psi(t)\Vert^2_{L^2_{5/2+\nu}}\!\! =\!\int (1+|y|^{5+2\nu})
\Big(\int\limits_0^t \dot\Psi(y,s)ds-\Psi_0(y)\Big)^2dy
\le 2d_0^2
+2t\!\int(1+|y|^{5+2\nu})dy\!\int\limits_0^t\dot\Psi^2(y,s)ds
\ee
Using the bounds (\ref{parameq}) we obtain
$$
|\dot v(s)|\le C(\ov v)\Vert\Psi(s)\Vert^2_{L^2}\le C(\ov v,d_0),\quad
|\dot b(s)|=|\dot c(s)+v(s)|\le|\dot c(s)|+1\le C(\ov v,d_0)
$$
\be\la{bs}
|b(s)|=|\int\limits_0^s\dot b(\tau) d\tau-b(0)|\le C(\ov v,d_0)s+|q_0|
\ee
Hence (\ref{add}) implies that
\beqn\nonumber
\dot\Psi^2(y,s)&=&
\Big[\dot b(s)\psi'(y+b(s),s)+\pi(y+b(s),s)-\dot v\pa_v\psi_v(y)\Big]^2\\
\la{enn}
&\le& C(\ov v,d_0)\Big((\psi'(y+b(s),s))^2+\pi^2(y+b(s),s)+(\pa_v\psi_v(y))^2\Big)\\
\nonumber
&\le& C(\ov v,d_0)\Big(e(y+b(s),s)+(\pa_v\psi_v(y))^2\Big)
\eeqn
Substituting (\re{enn}) into (\re{en}) and changing variables we
obtain by (\ref{wees}), (\ref{bs}) and  (\ref{U-int})
\beqn\nonumber
\Vert\Psi(t)\Vert^2_{L^2_{5/2+\nu}}\!\!\!&\le&\!\!\!
2d_0^2+ C(\ov v,d_0)t\int\limits_0^t\Big(\int(1+|x-b(s)|^{5+2\nu})
e(x,s)dx+C(\ov v)\Big)ds\\
\nonumber
\!\!\!&\le&\!\!\!2d_0^2+ C(\ov v,d_0)t^2+ C(\ov v,d_0)t\int(1+|x|^{5+2\nu})e(x,0)dx
\int\limits_0^t(1+s+|b(s)|)^{6+2\nu}ds\\
\nonumber
\!\!\!&\le&\!\!\!2d_0^2+ C(\ov v,d_0)t^2+C(\ov v,d_0)(1+t)^{8+2\nu}
\Big[\Vert X_0\Vert_{E_{5/2+\nu}}^2+U_0\Big]\le C(\ov v,d_0)(1+t)^{8+2\nu}
\eeqn


\begin{thebibliography}{28}

\bibitem{A}
S. Agmon, Spectral properties of Schr\"odinger operator and scattering theory,
{\em Ann.~Scuola Norm.~Sup.~Pisa}, Ser.~IV {\bf 2}, 151-218 (1975).

\bibitem{Bais82} F. A. Bais, Topological excitations in gauge theories; 
An introduction from the physical point of 
view, Springer Lecture Notes in Mathematics, vol. 926 (1982).

\bibitem{Bjorn98} 
F. Bj$\o$rn, Geometry, Particles, and Fields, Springer, NY, 1998.

\bibitem{BP}
V.S. Buslaev, G.S. Perelman,
Scattering for the nonlinear Schr\"odinger equations: states close to a soliton,
{\em St. Petersburg Math. J.} {\bf 4} (1993), no.6, 1111-1142.

\bibitem{BS}
V.S. Buslaev, C. Sulem, On asymptotic stability of solitary waves for nonlinear
Schr\"odinger equations,{\em Ann. Inst. Henri Poincar\'e, Anal. Non Lin\'eaire}
{\bf 20}(2003), no.3, 419-475.

\bibitem{Cu01}
S.~Cuccagna, Stabilization of solutions to nonlinear Schr\"odinger
equations, {\em Comm. Pure Appl. Math.} {\bf 54} (2001), 1110-1145.

\bibitem{Cu}
S. Cuccagna, On asymptotic stability in 3D of kinks for the $\phi^4$ model,
{\em Transactions of  AMS} {\bf 360} (2008), no. 5, 2581-2614.

\bibitem{jeka}
A. Jensen, T. Kato, Spectral properties of Schr\"odinger operators and time-decay
of the wave functions, {\em Duke Math.~J.} {\bf 46}, 583-611 (1979).

\bibitem{JN}
A. Jensen, G. Nenciu, A unified approach to resolvent expansions at thresholds,
{\em Reviews in Math. Phys.} {\bf 13} (2001), no. 6, 717-754.

\bibitem{HPW}
D.B. Henry, J.F. Perez, W.F. Wreszinski, Stability theory for solitary-wave solutions
of scalar field equations, {\em Comm. Math. Phys.} {\bf 85} (1982),
351-361.

\bibitem{IKV05}
V. Imaikin, A.I. Komech, B. Vainberg, On scattering of solitons for the Klein-Gordon
equation coupled to a particle, {\em Comm. Math. Phys.} {\bf 268} (2006), no.2, 321-367.
\bibitem{KZ07}
E.  Kirr, A.  Zarnesku,
On the asymptotic stability of bound states in 2D cubic Schro"dinger
equation, {\em Comm. Math. Phys.} {\bf 272} (2007), no. 2, 443-468.

\bibitem{KK10}
Komech A., Kopylova E., Weighted energy decay for
1D Klein-Gordon equation, {\em Comm. PDE} {\bf 35}, no.2, 353-374 (2010).

\bibitem{1dkg}
E. Kopylova,  On long-time decay for  Klein-Gordon equation,
ArXiv: 1009.2649

\bibitem{Li}  J.L. Lions,
 ``Quelques M\`ethodes de R\`esolution des Probl\'emes aux Limites
 non Lin\'eaires'', Paris, Dunod, 1969.

\bibitem{M}
M. Murata, Asymptotic expansions in time for solutions of
Schr\"odinger-type equations, {\em J.~Funct.~Anal.}{\bf 49}, 10-56 (1982).

\bibitem{MW96}
J. Miller, M. Weinstein,
Asymptotic stability of solitary waves for the regularized long-wave equation,
{\em Comm. Pure Appl. Math.} {\bf 49}  (1996),  no. 4, 399-441.

\bibitem{PW94}
R.L. Pego, M.I. Weinstein, Asymptotic stability of solitary waves,
{\em Commun. Math. Phys.} {\bf 164} (1994), 305-349.

\bibitem{PW} C.A.  Pillet, C.E. Wayne, Invariant manifolds for a class of dispersive,
Hamiltonian, partial differential equations,
{\em J. Differ. Equations} {\bf 141} (1997), No.2, 310-326.

\bibitem{Re} M. Reed,
``Abstract Non-Linear Wave Equations'', Lecture Notes
in Mathematics 507 (1976), Springer, Berlin.

\bibitem{RS3}
M. Reed, B. Simon, Methods of Modern Mathematical Physics, III,
Academic Press, 1979.

\bibitem{RSS05}
I. Rodnianski, W. Schlag, A. Soffer,
Dispersive analysis of charge transfer models,
{\em Commun. Pure Appl. Math.} {\bf 58} (2005), no. 2, 149-216.

\bibitem{SW1} A. Soffer, M.I. Weinstein,
Multichannel nonlinear scattering for nonintegrable equations,
{\em Comm. Math. Phys.} {\bf 133} (1990), 119-146.

\bibitem{SW2} A. Soffer, M.I. Weinstein,
Multichannel nonlinear scattering for nonintegrable equations. II. The case
of anisotropic potentials and data,
{\em J. Differential Equations} {\bf 98} (1992), no. 2, 376-390.

\bibitem{SW99}
A.  Soffer, M.I.  Weinstein, Resonances, radiation damping and instability
in Hamiltonian nonlinear wave equations,
{\em Invent. Math.}  {\bf 136} (1999), 9-74.

\bibitem{St78}  W.A. Strauss,
Nonlinear invariant wave equations, Lecture
Notes in Physics 73 (1978), Springer, Berlin, 197-249.

\bibitem{TY02}
T.-P. Tsai, H.-T. Yau, Asymptotic dynamics of nonlinear Schr\"odinger equations:
resonance-dominated and dispersion-dominated solutions,
{\em Commun. Pure Appl. Math.} {\bf 55} (2002), no.2, 153-216.

\end{thebibliography}
\end{document}